\documentclass{amsart}

\usepackage[all]{xy}
\usepackage{amsmath}
\usepackage{hyperref}
\usepackage{amsfonts,graphics,amsthm,amsfonts,amscd,latexsym}
\usepackage{epsfig}
\usepackage{flafter}
\usepackage{mathtools}

\makeatletter
\@namedef{subjclassname@2020}{%
  \textup{2020} Mathematics Subject Classification}
\makeatother

\calclayout

\hypersetup{
    colorlinks=true,    
    linkcolor=blue,          
    citecolor=blue,      
    filecolor=blue,      
    urlcolor=blue           
}
\usepackage{tikz}
\usetikzlibrary{graphs,positioning,arrows,shapes.misc,decorations.pathmorphing}

\tikzset{
    >=stealth,
    every picture/.style={thick},
    graphs/every graph/.style={empty nodes},
}

\tikzstyle{vertex}=[
    draw,
    circle,
    fill=black,
    inner sep=1pt,
    minimum width=5pt,
]
\usepackage[position=top]{subfig}
\usepackage[alphabetic, backrefs]{amsrefs}
\usepackage{amssymb}
\usepackage{color}

\usetikzlibrary{decorations.pathmorphing}
\tikzstyle{printersafe}=[decoration={snake,amplitude=0pt}]

\newcommand{\id}{\operatorname{id}}
\newcommand{\loc}{\operatorname{loc}}

\newcommand{\rank}{\operatorname{rank}}

\renewcommand{\qq}{\mathbb{Q}}
\newcommand{\zz}{\mathbb{Z}}
\newcommand{\nn}{\mathbb{N}}
\newcommand{\rr}{\mathbb{R}}
\newcommand{\cc}{\mathbb{C}}

\def\O#1.{\mathcal {O}_{#1}}			
\def\pr #1.{\mathbb P^{#1}}				
\def\af #1.{\mathbb A^{#1}}			
\def\ses#1.#2.#3.{0\to #1\to #2\to #3 \to 0}	
\def\xrar#1.{\xrightarrow{#1}}			
\def\K#1.{K_{#1}}						
\def\bA#1.{\mathbf{A}_{#1}}			
\def\bM#1.{\mathbf{M}_{#1}}				
\def\bL#1.{\mathbf{L}_{#1}}				
\def\bB#1.{\mathbf{B}_{#1}}				
\def\bK#1.{\mathbf{K}_{#1}}			
\def\subs#1.{_{#1}}					
\def\sups#1.{^{#1}}

\DeclareMathOperator{\Supp}{Supp}

\newcommand{\rar}{\rightarrow}

\usepackage{tikz}
\usetikzlibrary{matrix,arrows,decorations.pathmorphing}

  \newtheorem{introthm}{Theorem}
  \newtheorem*{jproperty}{Jordan property for ${\rm Gl}_n(\cc)$}

  \newtheorem{introcor}{Corollary}

  \newtheorem{theorem}{Theorem}[section]
  \newtheorem{lemma}[theorem]{Lemma}
  \newtheorem{proposition}[theorem]{Proposition}
  \newtheorem{corollary}[theorem]{Corollary}

  \newtheorem{definition}[theorem]{Definition}
  \newtheorem{example}[theorem]{Example}

  \newtheorem{question}[theorem]{Question}

\newtheorem{remark}[theorem]{Remark}

\theoremstyle{remark}

\numberwithin{equation}{section}

\begin{document}

\title[The Jordan property for local fundamental groups]{The Jordan property for local fundamental groups}

\thanks{
Part of this work was completed during a visit of RS to the Princeton University. 
RS would like to thank the Princeton University for the hospitality and the nice working environment, and Gabriele Di Cerbo for funding his visit.
RS is supported by the European Union's Horizon 2020 research and innovation programme under the Marie Sk\l{}odowska-Curie grant agreement No. 842071.
LB is partially supported by the DFG-Graduiertenkolleg GK1821 "Cohomological Methods in Geometry" at the University of Freiburg.
}

\author[L.~Braun]{Lukas Braun}
\address{Mathematisches Institut, Albert-Ludwigs-Universit\"at Freiburg, Ernst-Zermelo-Strasse 1, 79104 Freiburg im Breisgau, Germany}
\email{lukas.braun@math.uni-freiburg.de}

\author[S.~Filipazzi]{Stefano Filipazzi}
\address{
EPFL, SB MATH CAG,
MA C3 625 (B\^atiment MA),
Station 8,
CH-1015 Lausanne, Switzerland.
}
\email{stefano.filipazzi@epfl.ch}

\author[J.~Moraga]{Joaqu\'in Moraga}
\address{Department of Mathematics, Princeton University, Fine Hall, Washington Road, Princeton, NJ 08544-1000, USA.
}
\email{jmoraga@princeton.edu}

\author[R.~Svaldi]{Roberto Svaldi}
\address{EPFL, SB MATH-GE, MA B1 497 (B\^{a}timent MA), Station 8, CH-1015 Lausanne, Switzerland.}
\email{roberto.svaldi@epfl.ch}


\begin{abstract}
We show the Jordan property for regional fundamental groups of klt singularities of fixed dimension.
Furthermore, we prove the existence of effective simultaneous index one covers for $n$-dimensional klt singularities.
We give an application to the study of local class groups of klt singularities.
\end{abstract}

\maketitle
\setcounter{tocdepth}{1} 
\tableofcontents

\section{Introduction}

Throughout this paper, we work over the field $\mathbb{C}$ of complex numbers.
We work with germs of algebraic singularities over the complex numbers, unless otherwise stated:
when referring to \emph{a singularity} we will mean that.
The \emph{rank} of a group $G$ refers to the minimal number of generators of $G$.

The study of singularities is a classical and fundamental topic in algebraic geometry.
Over the field of complex numbers, one approach that dates back to the modern foundations of the subject is to understand the local topological structure of a given singularity.
Durfee,~\cite{Dur}, showed that the topology of a sufficiently small punctured neighborhood of an algebraic singularity stabilizes -- a result originally proved by Milnor for isolated singularities,~\cite{Mil}.
In view of these results, we can talk about the local fundamental group $\pi^{\loc}_1 (X,x)$ of a germ of an algebraic singularity $x\in X$ -- to formally define $\pi^{\loc}_1 (X,x)$, we will work with the system of analytic neighborhoods of an algebraic singularity, cf. Definition~\ref{def:loc.fund.gr}. 
We will clarify each time what kind of neighborhoods (whether algebraic or analytic) we work with.
It is then natural to wonder whether we can actually compute the local fundamental group of a given algebraic singularity, and, vice versa, how we can use this piece of information to characterize singularities.
For example, Mumford,~\cite{Mum61}, proved that the triviality of the local fundamental group completely characterizes smooth points on surfaces, a result which fails to hold in higher dimension, as  can be easily seen by considering the local fundamental group of an isolated hypersurface singularity.
In the same vein, one can also ask whether it is possible to characterize those groups that can be realized as local fundamental groups of a given algebraic singularity.
In~\cite{KK14}, Kapovich and Koll\'ar showed that already for 3-dimensional isolated singularities the local fundamental group can be arbitrarily complicated.

In this work, we will focus on those singularities related to the development of the Minimal Model Program, as they play a central role in modern birational geometry.
In particular, we will be interested in the class of so-called log canonical singularities, the largest class of singularities for which the Minimal Model Program is expected to work.
Despite being the most natural class of singularities that one would like to understand, Koll\'ar,~\cite{Kol11}, has shown that three dimensional log canonical singularities already display a wide variety of fundamental groups; in particular, these groups may be infinite.
In view of this, the next natural candidate is the subclass of Kawamata log terminal (in short, klt) singularities.
It is expected that this subclass of singularities is better behaved with respect to their topological structure:
Koll\'ar conjectured that the local fundamental group of a klt singularity is finite.

In our treatment we will consider the larger category of klt singularities $x \in (X, \Delta)$, where the boundary $\Delta$ is an effective Weil divisor with coefficients in $[0, 1)$.
Klt singularities are rational, and, from a topological viewpoint, strong evidence towards a simpler structure for the local fundamental group of klt singularities is provided by the contractibility of the dual complex of such singularities, see~\cite{dFKX}.
Works of Xu and Tian--Xu,~\cites{Xu14, TX17}, showed that by considering plt blow-ups, a special class of birational transformations  of a klt singularity, it is possible to reduce Koll\'ar's conjecture to an analogous conjecture on the finiteness of the orbifold fundamental group of the smooth locus of log Fano type varieties.

In dimension two, the finiteness of the fundamental group of a (possibly singular) Fano variety has been known for quite some time,~\cites{FKL93,GZ94,GZ95,KM99}, and~\cite{TX17} settled some partial results in dimension three.
Recently, the first-named author settled both conjectures in full generality,~\cite{Bra20}*{Theorems~1 and~2}, in fact, proving an even more comprehensive result: the finiteness of the fundamental group of the smooth locus of a neighborhood of a klt singularity $x\in (X, \Delta)$.
This variant of the fundamental group of a singularity is called the regional fundamental group $\pi_1^{\rm reg}(X,x)$ of a klt singularity.
It is not hard to see that the inclusion of the smooth locus induces a natural surjection $\pi_1^{\rm reg}(X,x) \twoheadrightarrow \pi_1^{\rm loc}(X,x)$.

\begin{introthm}
\label{intro-thm:braun}
(cf.~\cite{Bra20}*{Theorem 1 \& 2})
Let $(E,\Delta_E)$ be a log Fano pair, that is, $(E,\Delta_E)$ is klt and $-(K_E+\Delta_E)$ is ample.
Assume that $\Delta_E=\Delta'+\Delta''$, where the coefficients of $\Delta'$ belong to the set $\{1-\frac 1 n \ \vert \ n \in \mathbb{N}_{>0}\}$ and $\Delta''$ is effective.
Let $E^0 \subset E$ be be a big open subset on which $(E^0,\left.\Delta'\right|_{E^0})$ is log smooth.
Then, the orbifold fundamental group $\pi_1(E^0,\left.\Delta'\right|_{E^0})$ is finite.

Let $x\in (X,\Delta)$ be a klt singularity.
Assume that $\Delta=\Delta'+\Delta''$, where the coefficients of $\Delta'$ belong to the set $\{1-\frac 1 n \ \vert \ n \in \mathbb{N}_{>0}\}$ and $\Delta''$ is effective. Then the regional fundamental groups
$\pi_1^{\rm reg}(X,x)$ and $\pi_1^{\rm reg}(X,\Delta',x)$ are finite.
\end{introthm}

In the recent paper~\cite{XZ20}, the authors computed an effective bound on the order of the regional fundamental group in terms of the dimension of the singularity and of its normalized volume.

It is natural to ask whether it is possible to show that the local (resp., regional) fundamental group of a klt singularity carries any further structure besides its finiteness.
In~\cite{Bra20}*{Corollary~6}, it is shown that all finite groups can appear as regional fundamental groups of quotient singularities.
At a first glance, it would then seem that such question is untenable.
Nonetheless, from the point of view of the study of singularities, it makes sense not to consider the fundamental groups of all singularities at once, but rather to study their structure based on the dimension $n$ of the singularity.
This is even more evident in view of the Jordan property that we introduce below.

Following this line of thought, the main result of this paper is the following structure theorem showing that the regional fundamental group of a klt singularity of dimension $n$ contains a normal abelian subgroup whose rank and index can be bounded by a constant that only depends on the dimension $n$ of the singularity.

\begin{introthm}
\label{intro-theorem:Jordan-klt}
Let $n$ be a positive integer.
Then, there exists a constant $c=c(n)$, only depending on $n$,
that satisfies the following property. 
Let $x\in (X,\Delta)$ be a $n$-dimensional klt singularity.
Then, there is an exact sequence
\begin{equation}
\label{exact-sequence-Jordan}
\xymatrix{
1\ar[r] &
A\ar[r] &
\pi_1^{\rm reg}(X,\Delta,x) \ar[r]&
N\ar[r] &
1,
}
\end{equation}
where $A$ is a finite abelian group of rank at most $n$ and index at most $c(n)$.
Furthermore, since the regional fundamental group surjects onto the local fundamental group, we obtain an exact sequence 
\begin{equation}
\label{exact-sequence-Jordan-local}
\xymatrix{
1\ar[r]&
A'\ar[r]&
\pi_1^{\rm loc}(X,x) \ar[r]&
N'\ar[r]&
1,
}
\end{equation}
where $A'$ is a finite abelian group of rank at most $n$
and index at most $c(n)$.
\end{introthm}

Even though $A$ and $N$ are not uniquely determined, 
we will refer to them as the {\em abelian} and {\em non-abelian}
parts of the regional fundamental group, hence, the notation $A$ and $N$.
In the case of the local fundamental group, the above was conjectured by Shokurov.

\subsection*{The Jordan property}
Those klt singularities $x\in X$ that can be represented as a quotient $X=\mathbb{C}^n/G$, for some finite group $G\leqslant {\rm Gl}_n(\mathbb{C})$, represent an important and well-studied class with plenty of applications in many other branches of mathematics, most notably representation theory and mathematical physics.

It is a classical result of Jordan,~\cite{Jor1873}, that any such finite subgroup of the general linear group satisfies the following property.

\begin{jproperty} 
Let $n$ be a positive integer and $G \leqslant {\rm Gl}_n(\mathbb{C})$ be a finite group.
Then there exists a positive integer $a=a(n)$, only depending on the dimension $n$ of the ambient space, and a normal abelian subgroup $A \trianglelefteq G$ of index at most $a(n)$ in $G$.
\end{jproperty}

As $A$ is an abelian normal subgroup, the action of $A$ on $\mathbb{C}^n$ is diagonalizable and the quotient $T \coloneqq \mathbb{C}^n/A$ is a $\qq$-factorial toric singularity endowed with a finite morphism $T\rightarrow X$ of degree bounded by $a(n)$.
Moreover, the rank of $A$ is at most $n$.
As $G$ satisfies the Jordan property, it possible to control $\pi_1^{\rm reg}(X,x)$.
More precisely, we have a natural surjective homomorphism $G\rightarrow \pi_1^{\rm loc}(X,x)$, which actually becomes an isomorphism when considering $\pi^{\rm reg}_1(X,x)$, since we can assume that $G$ contains no pseudo-reflections.
Hence, the Jordan property holds for $\pi_1^{\rm loc}(X,x)$ when $x \in X$ is a $n$-dimensional quotient singularity, and the constant $a=a(n)$ above will only depend in this case on the dimension of $X$.

Theorem~\ref{intro-theorem:Jordan-klt} is clearly inspired by the Jordan property of the fundamental groups of quotient singularities just illustrated, and its scope is to show that an analogous principle holds for the local fundamental groups of a general $n$-dimensional klt singularity.
We shall say that a class $\mathcal{G}$ of finite groups satisfy the Jordan property if there exists a positive integer $d=d(\mathcal{G})$ such that for any $G \in \mathcal{G}$, there exists a normal abelian subgroup $A \trianglelefteq G$ of index at most $d$ in $G$.
Furthermore, if we can choose $A$ to have rank at most $r$,
we say that the class satisfy the Jordan property with rank $r$.
It should be clear to the reader that such scope has been achieved and that Theorem \ref{intro-theorem:Jordan-klt} can be rephrased in the following coincise form.

\begin{introcor}
\label{intro-cor:jordan}
Let $n$ be a positive integer.
Let $\mathcal{R}_n$ (resp., $\mathcal{L}_n$) be the class of regional fundamental groups $\pi_1^{\rm reg}(X, \Delta, x)$ (resp., local fundamental groups $\pi_1^{\rm loc}(X, \Delta, x)$) of klt singularities $x \in (X, \Delta)$ of dimension $n$.
Then $\mathcal{R}_n$ (resp., $\mathcal{L}_n$) satisfies the Jordan property with rank $n$.
\end{introcor}

To prove the Jordan property for the regional and local fundamental groups of a klt singularity, we following the global-to-local philosophy: indeed, we first prove that the Jordan property holds for the orbifold fundamental group of a (weakly) log Fano pair.

\begin{introthm}
\label{intro-thm:JP.orb}
(cf. Theorem~\ref{thm:finiteness-weakly-LF})
Let $(E,\Delta)$ be a pair of weakly log Fano type.
Let $E^0$ be a big open subset on which $(E^0,\Delta^0)$ is log smooth,
where $\Delta^0$ is the restriction of $\Delta$ to $E^0$.
Then $\pi_1(E^0,\Delta^0)$ is finite.
Furthermore, the class of fundamental groups of
the log smooth locus of weakly log Fano type pairs of dimension $n$ satisfy the Jordan property with rank $n$.
\end{introthm}

In particular, the above Theorem immediately shows that the fundamental group of the smooth locus of a normal projective variety with klt singularities and ample anticanonical class satisfies the Jordan property.
Theorem~\ref{intro-thm:JP.orb} is deduced from the recent result of Prokhorov--Shramov,~\cite{PS14}, on the Jordan property for the birational automorphism groups of Fano varieties; 
Prokhorov--Shramov's result in turn relies on Birkar's proof of the BAB Conjecture,~\cites{Bir16a, Bir16b}.

Once the Jordan property is understood in the global framework, the main core of the paper is dedicated to the study of how the Jordan property can be lifted from the orbifold fundamental group of the Koll\'ar component of a plt blow-up to a suitable neighborhood of a klt singularity.
This is carried out in \S~\ref{sec:whitney}-\ref{sec:proof-jordan}.

\subsection*{Applications}

Theorem~\ref{intro-theorem:Jordan-klt} immediately implies the following bound on the number of generators of the regional fundamental group of a $n$-dimensional klt singularity.

\begin{introcor}\label{introcor-1}
Let $n$ be a positive integer.
There exists a constant $b=b(n)$ only depending on $n$ that satisfies the following property.
Let $x\in (X,\Delta)$ be a $n$-dimensional klt singularity,
then the regional fundamental group 
$\pi_1^{\rm reg}(X,\Delta,x)$ is generated by at most $b(n)$ elements.
Furthermore, the number of generators of order larger than $b(n)$ is at most $n$.
\end{introcor}

Moreover, we can define the {\em rank} of a $n$-dimensional klt singularity $x\in (X, \Delta)$ to be the minimum rank of an abelian subgroup $A$ of index at most $c(n)$ in $\pi^{\rm reg}_1(X,x)$, where $c(n)$ is the constant appearing in the statement of the theorem.
We shall say that a class of singularities has rank $k$, if any singularity in such class has rank bounded by $k$.
For instance, if a class of klt singularities has rank zero, 
then the order of the local fundamental group of these singularities is bounded by a constant which only depend on the dimension.
In Example~\ref{ex:exceptional-singularities-example}, we show that the class of exceptional $n$-dimensional singularities has rank one, 
while, for any $\epsilon>0$, the class of $\epsilon$-log canonical exceptional $n$-dimensional singularities has rank zero.

As any Weil divisor of finite Cartier index induces a quasi-\'etale cover in a neighborhood of $x \in X$, Theorem~\ref{intro-theorem:Jordan-klt} implies the existence of a bound on the number of generators of the local class group of a $n$-dimensional klt singularities.

\begin{introcor} \label{thm_class_group}
Let $n$ be a positive integer.
There exists a positive integer $b=b(n)$, only depending on $n$, satisfying the following.
Let $x\in (X,\Delta)$ be a $n$-dimensional klt singularity,
then the torsion subgroup ${\rm Cl}(\O X,x.)_{\rm tor}$ of the local class group ${\rm Cl}(\O X,x.)$ is generated by at most $b(n)$ elements.
Moreover, at most $n$ generators of ${\rm Cl}(\O X,x.)_{\rm tor}$ can have order larger than $b(n)$.

The same statement holds for ${\rm Cl}(\O X,x.^{\rm hol})_{\rm tor}$, the torsion part of the analytic local class group.
\end{introcor}

In view of Theorem~\ref{intro-theorem:Jordan-klt}, it is natural to wonder whether and how the existence of a large abelian subgroup $A' \leqslant \pi^{\rm loc}_1(x\in X)$ may be reflected in the geometry of the singularity $x\in X$.
In~\cite{Mor20}, the third author shows that, using the theory of complements and the $G$-equivariant MMP for Fano-type surfaces, the existence of such a subgroup $A'$ implies that the singularity can be deformed to a toric quotient singularity.
In a similar vein, we study the existence of simultaneous index one covers of $\qq$-Cartier divisors on $n$-dimensional klt singularities.
In this direction, we prove that such covers can be obtained with a cover of bounded index, followed by at most $n$ cyclic covers.

\begin{introthm}
\label{thm:effective-index-one-cover}
Let $n$ be a positive integer.
There exists a constant $b'=b'(n)$,
only depending on $n$, that satisfies the following property.
Let $x\in X$ be a $n$-dimensional klt singularity.
Then, there exists a sequence of finite quasi-\'etale covers
\[
\xymatrix{
X &
\ar[l] X_1 &
\ar[l] X_2 &
\ar[l] \dots &
\ar[l] X_{k},
}
\]
satisfying the following conditions:
\begin{itemize}
  \item $k \leq n+1$;
    \item for any $1 \leq i \leq k$, $x\in X$ has a unique preimage $x_i\in X_i$;
    \item every integral $\qq$-Cartier divisor on $X_k$ is Cartier;
    \item the morphism $X\longleftarrow X_1$ has degree at most $b'(n)$; and 
    \item for any $1 \leq i \leq k-1$, the morphism $X_i \longleftarrow X_{i+1}$ is a cyclic cover.
\end{itemize}
In particular, $x_{k}\in X_{k}$ is a canonical singularity.
\end{introthm}

As the statement of Theorem \ref{thm:effective-index-one-cover} is local in nature, we may shrink $X$ around $x$ if necessary without mentioning it.
The result in Theorem~\ref{thm:effective-index-one-cover} is a refinement of~\cite{GKP}*{Theorem~1.10}
which is a statement about algebraic (equivalently, Zariski open) neighborhoods.
As explained in the proof of Theorem~\ref{thm:effective-index-one-cover}, the statement of the theorem works also in the case of an analytic germ of an algebraic singularity.

As already mentioned above, in~\cites{Xu14, TX17}, the study of the local fundamental group of a klt singularity $x \in (X, \Delta)$ was carried out by comparing the topology of a punctured neighborhood of the singularity with the topology of the Koll\'ar component $E$ extracted by a plt blow-up $Y \to X$ of the singularity.  
Such divisor $E$ carries a natural log Fano pair structure $(E, \Delta_E)$, by the adjunction formula, that is,  $\Delta_E$ is the different of $\Delta$ on $E$.
As $E$ may contain a component of the singular locus of $Y$ of codimension 2, it is necessary to consider the orbifold fundamental group $\pi_1(E^0, \Delta_E\vert_{E^0})$, where $E^0 \subset E$ is a Zariski open set where the pair is log smooth. 
Hence, the divisor $E$ plays a central role in the understanding of the fundamental group of a singularity.
At the generic point of $E$, $X$ can be thought from a topological viewpoint as a disk bundle over $E$, and the class of the loop generating the fundamental group of such a disk is one of the main ingredients in our proof, see Remark~\ref{rem:loop-plt} for a more precise definition.
Rather surprisingly, we can prove that this loop lies in the center of the regional fundamental group, a result of independent interest from the topological point of view.

\begin{introthm}\label{introthm:center}
Let $x\in(X,\Delta)$ be a klt singularity. Let $\pi\colon Y\rightarrow X$ be a plt blow-up
and $\gamma_E$ the loop around the exceptional divisor $E$.
Then, $\gamma_E \in Z(\pi_1^{\rm reg}(X,\Delta,x))$.
\end{introthm}

Moreover, it is useful to understand the center of the regional fundamental group of a klt singularity, in order to find the largest possible normal abelian subgroup of $\pi_1^{\rm reg}(X,\Delta,x)$ satisfying the properties of Theorem~\ref{intro-theorem:Jordan-klt}.
Indeed, in the statement of the theorem we can always replace $A$ by the larger abelian subgroup generated by $A$ and the center of the regional fundamental group.

\subsection*{Structure of the paper}
The structure of the paper is as follows: in \S\ref{sec:prelim}, we introduce the basic definitions and preliminary results that will be used in the rest of the paper;
in \S\ref{sec:jordan}, we show that the Jordan property holds for the orbifold fundamental group of $n$-dimensional log Fano pairs with standard coefficients.
In \S\ref{sec:whitney}, we recall some results on the Whitney stratification of algebraic varieties and show that given a plt blow-up of a klt singularity, there exists a loop in $\pi^{\rm reg}_1(X,x)$ that is naturally associated to this construction and that is contained in the center $Z(\pi^{\rm reg}_1(X,x))$ of the regional fundamental group.
In \S\ref{sec:proof-jordan}, we prove the Jordan property for the regional fundamental group of $n$-dimensional klt singularities.
In \S\ref{sec:simultaneous-index}, we prove the first application to effective simultaneous index one covers.
In \S\ref{sec:local-class}, we prove the second application to the local class group of $n$-dimensional klt singularities.
Finally, in \S\ref{sec:ex}, we give some examples and further questions.

\subsection*{Acknowledgements} 
The authors would like to thank 
Javier Carvajal-Rojas,
J\'anos Koll\'ar, 
Mirko Mauri,
Thomas Peternell,
Yuri Prokhorov, and
Chenyang Xu for many useful comments.
We wish to thank the anonymous referee for useful comments and suggestions that helped the authors improve the exposition.

Part of this work was completed during a visit of RS to the Princeton University. 
RS would like to thank the Princeton University for the hospitality and the nice working environment, and Gabriele Di Cerbo for funding his visit.
RS is supported by the European Union's Horizon 2020 research and innovation programme under the Marie Sk\l{}odowska-Curie grant agreement No. 842071.
LB is partially supported by the DFG-Graduiertenkolleg GK1821 "Cohomological Methods in Geometry" at the University of Freiburg.

\section{Preliminaries}
\label{sec:prelim}
In this section, we collect some preliminary results about 
local fundamental groups, the minimal model program, 
and the Jordan property.

\begin{definition}
{\em
Let $X$ be a quasi-projective algebraic variety and
$\Delta$ an effective $\qq$-divisor on $X$, we say that 
\begin{enumerate}
\item 
$(X,\Delta)$ is a \emph{pair} if $K_X+\Delta$ is $\qq$-Cartier.
\item 
$(X, \Delta)$ has \emph{standard coefficients} if the coefficients of $\Delta$ belong to the set $\lbrace 1-\frac{1}{n}| n \in \nn \rbrace$.
\item
A pair $(X,\Delta)$ is {\em log smooth} if $X$ is smooth and $\Delta$ has simple normal crossing support.
\end{enumerate}
}
\end{definition}

Given a projective birational morphism 
$\pi \colon Y \rightarrow X$ from a normal variety $Y$ 
and a prime divisor $E\subset Y$, the {\em log discrepancy} $a_E(X,\Delta)$ of $(X,\Delta)$ at $E$ is
\[
a_E(X,\Delta) \coloneqq
1-{\rm coeff}_E(K_Y-\pi^*(K_X+\Delta)),
\]
where the canonical divisor $\K Y.$ satisfies $\pi_*(\K Y.)= \K X.$.

\begin{definition}
{\em
Let $(X,\Delta)$ be a log pair, we say that
\begin{enumerate}
    \item 
$(X,\Delta)$ is {\em log canonical} (resp. {\em Kawamata log terminal}) if for any birational morphism $\pi \colon Y\to X$ and any prime divisor $E$ on $Y$, $a_E(X,\Delta) \geq 0$ (resp., $a_E(X,\Delta) >0$).
    \item
$(X,\Delta)$ is \emph{purely log terminal} (in short, plt), if for every divisor $E \subset Y$ that is exceptional over $X$, we have $a_E(X,\Delta) > 0$.
\end{enumerate}
}
\end{definition}

As usual, we will abbreviate log canonical (resp., Kawamata log terminal) with lc (resp., klt).

If $(X, \Delta)$ is a plt pair, then every component of $\lfloor \Delta \rfloor$ is a normal prime divisor in $X$ satisfying $a_E(X,\Delta)=0$, and any two such divisors must be disjoint, see~\cite{KM98}*{Proposition~5.51}.

The following definition, first introduced in~\cite{P98}*{Definition~2.1}, identifies certain special types of birational extractions that naturally yield a plt pair over a klt singularity.

\begin{definition}
{\em 
Let $(X,\Delta)$ be a klt pair and $x\in X$ be a closed point.
A projective birational morphism
$\pi \colon Y \rightarrow X$ of normal varieties is a {\em purely log terminal blow-up} (or plt blow-up for short) if 
\begin{enumerate}
    \item 
$\pi$ extracts a unique prime divisor $E$ with $\pi(E)=x$;
    \item 
the pair $(Y,\Delta_Y+E)$ is plt, where $\Delta_Y$ is the strict transform of $\Delta$ on $Y$; and
    \item 
$-E$ is ample over $X$.    
\end{enumerate}
}
\end{definition}

The exceptional component $E$ of a purely log terminal blow-up is also indicated in the literature with the name of \emph{Koll\'ar component}, cf., for example,~\cite{Xu14}*{p. 412}.
The existence of plt blow-ups is now a well established fact, see~\cite{Xu14}*{Lemma 1}.

\begin{definition}{\em
Let $(X,\Delta)$ be a projective pair.
The pair $(X,\Delta)$ is {\em log Fano} if it is klt and $-(\K X. + \Delta)$ is ample.
The pair $(X,\Delta)$ is {\em weakly log Fano} if it is klt and $-(K_X+\Delta)$ is nef and big.
The variety $X$ is said to be of {\em Fano type} if there exists a boundary $\Delta$ on $X$ so that
$(X,\Delta)$ is log Fano.
A log pair $(X,\Delta)$ is said to be of 
{\em weakly log Fano type} if there exists a boundary $\Delta'\geq \Delta$ so that $(X,\Delta')$ is weakly log Fano.
}
\end{definition}

We remind the reader that varieties of weakly log Fano type are Mori dream spaces, see~\cite{BCHM}*{Corollary 1.4.3}.

\begin{definition}
{\em 
Let $f \colon X \rar Z$ be a morphism between normal varieties.
The morphism $f$ is said to be \emph{quasi-\'etale} if it is quasi-finite and \'etale in codimension one.
}
\end{definition}

\begin{definition}
{\em 
Let $f \colon X \rar Z$ be a morphism between normal varieties.
The morphism $f$ is said to be \emph{Galois} if it is finite, surjective, and there exists a finite group $G \leqslant \mathrm{Aut}(X)$ so that $f$ coincides with the quotient map $X \rar X/G$.
}
\end{definition}

\begin{definition}
{\em 
Let $x\in X$ be a closed point on a complex algebraic variety.
The {\em local fundamental group} $\pi_1^{\rm loc}(X,x)$ of $X$ at $x$
is the inverse limit of the fundamental groups of
punctured analytic neighborhoods of $x$ on $X$.
The {\em regional fundamental group}
$\pi^{\rm reg}(X,x)$
is the inverse limit of the fundamental groups of the intersection of analytic neighborhoods of $x$ in $X$ with the smooth locus of $X$.
}
\end{definition}

It is not hard to see that the inverse limit in the definition of $\pi_1^{\rm loc}(X,x)$ (resp.
$\pi_1^{\rm reg}(X,x)$) is computed by some analytic open neighborhood
$x\in U\subset X$, so that $\pi_1^{\rm loc}(X,x)\simeq \pi_1(U^0)$
(resp. $\pi_1^{\rm reg}(X,x)=\pi_1(U\cap X^{\rm reg})$),
where $U^0=U\setminus \{x\}$.

Now, we turn to introduce the fundamental group of a log pair, generalizing the usual fundamental group of a topological space. 
Essentially, the fundamental group of a log pair will be the fundamental group of an underlying orbifold structure which is dominated by the pair structure.
First, we recall the notion of orbifold fundamental group.

\begin{definition}
\label{def:orb.fund.gp}
{\em
Let $(X,\Delta)$ be a log smooth pair with standard coefficients,
i.e.,
we can write 
\[
\Delta=\sum_{i=1}^d \left(1-\frac{1}{n_i}\right)P_i,
\]
where the $P_i$ are pairwise distinct prime divisors.
The {\em orbifold fundamental group} of $(X,\Delta)$, denoted by $\pi_1(X,\Delta)$, is the fundamental group of the Deligne--Mumford stack
$(\mathcal{X},\mathcal{D})$ whose coarse moduli space is $(X,\Delta)$.
Let $(X,\Delta)$ be a log pair. 
We define the {\em standard approximation} of $\Delta$, denoted by $\Delta_s$, to be the largest effective divisor with standard coefficients so that $\Delta_s\leq \Delta$.
Note that $\Delta_s=0$ if and only if all the coefficients of $\Delta$ are smaller than $\frac{1}{2}$.
We define the fundamental group of a log smooth pair $(X,\Delta)$ to be
$\pi_1(X,\Delta) \coloneqq \pi_1(X,\Delta_s)$, where the latter is the orbifold fundamental group.
}
\end{definition}

\begin{remark}
\label{rem:loop}
{\em
Given a log smooth pair $(X, \Delta)$ with standard coefficients, $\Delta=\sum_{i=1}^d (1-\frac{1}{n_i})P_i$, for each prime divisor $P_j$ in the support of $\Delta$ with coefficient $1-\frac{1}{n_j}$, 
it is possible to associate a class $\gamma_j \in \pi_1(X \setminus {\rm Supp}(\Delta))$ to the pair $(P_j, n_j)$:
the class $\gamma_j$ is the class of a loop in the normal circle bundle of $P_j \setminus (\cup_{i \neq j} P_i) \subset X \setminus {\rm Supp}(\Delta)$ generating the fundamental group of the fibre of the bundle structure.
We refer to any loop in the class $\gamma_j$ as {\it the loop around the prime divisor $P_j$}; abusing notation, we will use the same expression also for the class $\gamma_j \in \pi_1(X \setminus {\rm Supp}(\Delta))$, when the context permits it.
The inclusion of $X \setminus {\rm Supp}(\Delta)$ in the coarse moduli space $(X, \Delta)$ of the orbifold $(\mathcal{X}, \mathcal{D})$ induces a surjective morphism
\[
\pi_1(X\setminus {\rm Supp}(\Delta)) \rightarrow \pi_1(X,\Delta) \rightarrow 1,
\]
whose kernel is the normal subgroup generated by the classes $\gamma_1^{n_1},\dots,\gamma_k^{n_k}$, see~\cite{Cam11}*{12.2}.
}
\end{remark}

\begin{definition}
\label{def:loc.fund.gr}
{\em 
Let $(X,\Delta)$ be a log pair 
and $x\in X$ be a closed point.
For each prime component $P$ 
of $\Delta_s$ through $x$,
we denote by $\gamma_P$ the loop around $P$ and by $n_P$ the positive integer so that $\gamma_P$ appears with coefficient $1-\frac{1}{n_P}$ on $\Delta_s$.
We define the group
$\pi_1(U,\Delta|_U)$ to be
the quotient 
of $\pi_1(U\setminus \Supp(\Delta))$ by the normal subgroup generated by the elements $\gamma_P^{n_P}$.
We can define the {\em local fundamental group of a pair}
$\pi_1^{\rm loc}(X,\Delta,x)$ of $X$ at $x$ to be the inverse limit of $\pi_1(U,\Delta|_U)$ where $U$ runs over all punctured analytic neighborhoods of $x$ on $X$.
Analogously, we can define the {\em regional fundamental group of a pair}
$\pi_1^{\rm reg}(X,\Delta,x)$ to be the inverse limit of 
$\pi_1(U,\Delta_U)$ where $U$ runs over all analytic neighborhoods of $x$ in $X$ intersected with the smooth locus of $X$.
}
\end{definition}

Let us note that the finite covers corresponding to normal subgroups of finite index of $\pi_1(X,\Delta)$ may ramify in codimension one, 
but their ramification is only concentrated along the support of $\Delta$ and the order is controlled by the coefficients of $\Delta$.
Indeed, by the Riemann--Hurwitz formula, such covers ramify with order at most $n_i$ over those components of $\Delta$ of coefficient $1-\frac{1}{n_i}$.

In \cite{Pop}*{Definition~2.1}, Popov defined the Jordan property for groups.
Here, we generalize that definition to a collection of groups, in the following way.

\begin{definition}
{\em 
Let $\mathcal{G}$ be a class of groups.
Then $\mathcal G$ is said to satisfy the Jordan property, if there exists a constant $k=k(\mathcal G)$, so that for every finite group $G\in \mathcal{G}$, $G$ contains a normal abelian subgroup of index at most $k$.
Furthermore, if the normal abelian subgroup has rank at most $r$, we say that the class $\mathcal{G}$ satisfy the Jordan property with rank $r$.
}
\end{definition}

\begin{lemma} \label{lemma_cyclic}
Let $G \subset GL_n(\cc)$ be a finite subgroup.
Assume that there exists a hyperplane $V \subset \cc^n$ such that $G$ acts trivially on $V$.
Then, $G$ is a cyclic group.
\end{lemma}

\begin{proof}
As $G$ is finite, hence reductive, there exists a 1-dimensional subspace $W \subset \cc^n$ such that $\mathbb C^n = V \oplus W$ and the action of $G$ on $W$ is faithful, as $G \subset GL_n(\cc)$, which then proves the claim.
\end{proof}

\begin{corollary} \label{group_fix_divisor}
Let $X$ be a normal variety, and let $D$ be a prime divisor on $X$.
Let $G \subset \mathrm{Aut}(X)$ be a finite subgroup fixing $D$ pointwise.
Then, $G$ is a cyclic group.
\end{corollary}

\begin{proof}
Let $x \in D \subset X$ be a closed point such that both $X$ and $D$ are smooth at $x$.
Then, by \cite{FZ05}*{Lemma~2.7.(b)}, $G$ acts faithfully on $T_xX \simeq \cc^n$.
By assumption, $G$ acts trivially on $T_xD \subset T_xX$.
Then, by Lemma~\ref{lemma_cyclic}, $G$ is a cyclic group.
\end{proof}

\begin{lemma} \label{lemma_rank}
Let $A$ be an abelian group of finite rank, and let $B$ be a subgroup.
Then, $\mathrm{rank}(B) \leq \mathrm{rank}(A)$.
\end{lemma}

\begin{proof}
The statement follows at once by tensoring the natural injection $B \hookrightarrow A$ with $\qq$.
\end{proof}

\section{Jordan property of global fundamental groups}
\label{sec:jordan}

In this section, we prove the Jordan property for the fundamental group of the log smooth locus of Fano type varieties.

The following theorem is well-known.
It was proved in~\cite{PS14}, under the assumption that the BAB conjecture holds. The BAB conjecture was proved in~\cites{Bir16a,Bir16b}.

\begin{proposition}[\cite{PS14}*{cf. Theorem 1.8}]\label{prop:Jordan-Bir}
Let $n$ be a positive integer.
There exists a constant $d=d(n)$,
only depending on $n$, 
satisfying the following.
Let $X$ be a $n$-dimensional rationally connected variety of dimension $n$.
Then, for every finite subgroup $G\leqslant {\rm Bir}(X)$,
there exists a normal abelian subgroup $A\leqslant G$ of rank $n$ and index at most $d(n)$.
\end{proposition}

The bound on the rank of $A$ is not explicitly mentioned in~\cite{PS14}.
However, its existence follows directly from the proof of~\cite{PS16}*{Theorem~4.2}, as we illustrate below.

\begin{proof}
The group $G$ acts on $X$ as a birational automorphism group.
By~\cite{Sum74}*{Theorem~3}, we may take a regularization of the action of $G$ on $X$,
i.e., a birational model on which $G$ acts as an automorphism group.
We denote by $X'$ such regularization.
Furthermore, we may take a $G$-equivariant resolution of singularities of $X'$ (e.g., see~\cite{AW97}*{Theorem~0.1}).
We call this resolution of singularities $X''$.
Hence, we have a rationally connected smooth projective variety 
on which $G$ acts as an automorphism group.
By~\cite{PS16}*{Theorem~4.2} and~\cite{Bir16b}*{Theorem~1.1}, we know that there exists a constant $r(n)$, only depending on $n$,
so that $G$ contains a subgroup $F \leqslant G$ of index at most $r(n)$,
and $F$ has a fixed point on $X''$.
We denote such point by $x''$.
In particular, $F$ acts faithfully on the tangent space of $X''$ at $x''$ (e.g., see \cite{PS14}*{proof of Theorem~4.2}),
so $F$ acts faithfully on $\cc^n$.
Hence, $F$ is a finite subgroup of ${\rm GL}_n(\cc)$.
In particular, it contains an abelian subgroup $A_0$ of rank at most $n$ 
and whose index in $F$ is at most $s(n)$, where $s(n)$ only depends on $n$.
We conclude that $A_0$ is an abelian subgroup
of $G$ of rank at most $n$
and index at most $t(n)\coloneqq r(n)s(n)$ (which still only depends on $n$).
Let $A$ be the maximal normal subgroup of $G$ contained in $A_0$, that is, the normal core of $A_0$ in $G$.
Then, by standard properties of the normal core, the index of $A$ on $G$ is at most $d(n)=t(n)!$.
Furthermore, since $A_0$ is abelian with rank at most $n$, it follows that $A$
has rank at most $n$.
\end{proof}

The following proposition proves the Jordan property for the fundamental group of the smooth locus of $n$-dimensional log Fano pairs.

\begin{proposition}\label{prop:Jordan-global-fundamental}
Let $n$ be a positive integer.
Then, there exists a constant $e(n)$, only depending on $n$, satisfying the following.
Let $(E,\Delta)$ be a $n$-dimensional log Fano pair with $\Delta$ having standard coefficients.
Let $E^0$ be a big open subset of $E$ on which $(E,\Delta)$ is log smooth.
Then, there exists a normal abelian subgroup
$A\leqslant \pi_1(E^0,\Delta^0)$ of rank at most $n$ and index at most $e(n)$.
\end{proposition}

\begin{proof}
Replacing $(E,\Delta)$ with a small $\qq$-factorialization $\phi \colon E^{\mathbb Q} \rar E$ does not change $(E^0,\Delta^0)$.
Also, since the morphism $\phi$ is birational, we may find an effective divisor $\Delta'$ on $E^{\mathbb Q}$ so that $-\Delta'$ is $\phi$-ample.
Thus, for $0<\epsilon \ll 1$, the pair $(E^{\mathbb Q},\Delta^{\mathbb Q}+\epsilon\Delta')$ is log Fano, where we have $\K E^{\mathbb Q}. + \Delta^{\mathbb Q}=\phi^*(\K E. + \Delta)$.
In particular, the hypotheses of~\cite{TX17}*{Proposition~3.6} are satisfied.
In the following, for the sake of brevity, we will assume that $E=E^{\mathbb Q}$ and $\Delta'=0$.
This assumption does not change the nature of the proof and will allow us to apply~\cite{TX17}*{Proposition~3.6} directly, without introducing extra notation.

By~\cite{Bra20}*{Theorem 2}, we know that $\pi_1(E^0,\Delta^0)$ is finite.
We will denote $G\coloneqq\pi_1(E^0,\Delta^0)$.
Let $(Z^0,\Delta_{Z^0})$ be the orbifold universal cover of $(E^0,\Delta^0)$.
By~\cite{TX17}*{Proposition~3.6}, we know that there exists a
normal analytic space $Z$, 
with a properly discontinuous action of $G$, 
so that $Z^0\subset Z$ is a big open subset, 
and the quotient of the log pair 
$(Z,\Delta_{Z})$ by $G$ equals $(E,\Delta)$.

Since $E$ is projective and $\pi \colon Z\rightarrow E$ is finite,
we conclude that $Z$ is itself a projective variety.
Moreover, we may write
\[
\pi^*(K_E+\Delta)=K_Z+\Delta_Z.
\]
By the Riemann--Hurwitz formula and the properties of universal covers, we have that $\Delta_Z \geq 0$.
Hence, we conclude that $Z$ has klt sigularities
and $-(K_Z+\Delta_Z)$ is an ample divisor.
Thus, $Z$ is a projective variety, and $(Z,\Delta_Z)$ is log Fano.
In particular, $Z$ is rationally connected.

We claim that for every element $g\in G$, 
$g$ induces a birational automorphism of $Z$.
Indeed, let $Z^1$ be $Z^0\setminus \pi^{-1}(D)$.
Then, $G$ is the group of deck transformations of the \'etale cover 
$Z^1 \to E^\circ \setminus D$.
In particular, any element of $G$ acts as an isomorphism of $Z^1$.
In particular, we have the following inclusions
\[
G\leqslant {\rm Aut}(Z^1)  \leqslant {\rm Bir}(Z).
\]
Now, we are in the setting of Proposition~\ref{prop:Jordan-Bir}.
Indeed, we have that $\pi_1(E^0,\Delta^0)=G$ 
and $G\leqslant {\rm Bir}(Z)$ is a finite subgroup,
where $Z$ is a $n$-dimensional rationally connected variety.

Indeed, applying Proposition~\ref{prop:Jordan-Bir}, we know that there is an exact sequence
\[
1\rightarrow A \rightarrow \pi_1(E^0,\Delta^0) \rightarrow N \rightarrow 1, 
\]
where $A$ is a finite abelian group of rank at most $n$
and index at most $d(n)$, where $d(n)$ is as in the statement of Proposition~\ref{prop:Jordan-Bir}.
Thus, we can choose $e(n)=d(n)$.
\end{proof}

We finish this section, by generalizing the above results to weakly log Fano type varieties.
This result is obtained by running a minimal model program and reducing the statement to the above proposition.

\begin{theorem}\label{thm:finiteness-weakly-LF}
Let $(E,\Delta)$ be a pair of weakly log Fano type.
Let $E^0$ be a big open subset on which $(E^0,\Delta^0)$ is log smooth,
where $\Delta^0$ is the restriction of $\Delta$ to $E^0$.
Then $\pi_1(E^0,\Delta^0)$ is finite.
Furthermore, the class of fundamental groups of
the log smooth locus of weakly log Fano type pairs of dimension $n$ satisfy the Jordan property with rank $n$.
\end{theorem}

\begin{proof}
As in the proof of Proposition~\ref{prop:Jordan-global-fundamental},
we may pass to a small $\qq$-factorialization of $E$.
This does not change the fundamental group of the log smooth locus.
First, we prove that $\pi_1(E^0,\Delta^0)$ is finite.
By assumption, $(E,\Delta)$ is of weakly log Fano type.
Hence, there exists $\Delta'\geq \Delta$ with
$(E,\Delta')$ klt and $-(K_E+\Delta')$ nef and big.
Since $E$ is a Mori dream space, $-(K_E+\Delta')$ is a semi-ample divisor.
Therefore, we can find $B\sim_\qq -(K_E+\Delta')$ so that $(E,\Delta'+B)$ is klt.
Let $\Delta_s$ denote the standard approximation of $\Delta$, and let $\Delta_s^0$ be its restriction to $E^0$.
By definition, we have that
\[
\pi_1(E^0,\Delta^0) \coloneqq
\pi_1(E^0,\Delta_s^0).
\]
Since $\Delta_s \leq \Delta$ and $E$ is $\qq$-factorial, it follows that $-(K_{E}+\Delta_s)$ is a big $\qq$-Cartier divisor.
Thus, we may run a minimal model program for the divisor $-(K_E+\Delta_s)$.
It consists in a sequence of flips and divisorial contractions
\[
E \eqqcolon E_0 \dashrightarrow E_1 \dashrightarrow E_2 \dashrightarrow \dots \dashrightarrow E_k.
\]
It terminates since $E$ is a Mori dream space.
We denote by $\Delta_{E_k}$ the push-forward of $\Delta$ to $E_k$.
By construction, we have that
$-(K_{E_k}+\Delta_{E_k,s})$ is nef and big.
We claim that $(E_k,\Delta_{E_k,s})$ is klt.
Indeed, we have that 
$\Delta_{E_k,s}\leq \Delta'_k+B_k$,
where $\Delta'_k$ (resp. $B_k)$ is the strict transform of $\Delta'$ (resp. $B$) on $E_k$.
Furthermore, all the steps $E_i\dashrightarrow E_{i+1}$ are log crepant for the log pair
$(E,\Delta'+B)$.
Since $(E,\Delta'+B)$ is klt, then so is $(E_k,\Delta'_k+B_k)$.
Thus, as we have $\Delta_{E_k,s} \leq \Delta'_k+B_k$, 
$(E_k,\Delta_{E_k,s})$ is klt as well.
By~\cite{Bra20}*{Theorem 1}, we conclude that
$\pi_1(E^0_k,\Delta^0_{E_k,s})$ is finite.
As usual, the super-script zero means that we are working over a big open subset on which the pair is log smooth.

We claim that, we have a sequence of surjections
\[
\pi_1(E^0_k,\Delta^0_{E_k,s})\rightarrow
\pi_1(E_{k-1}^0,\Delta^0_{E_{k-1},s}) \rightarrow
\dots \rightarrow
\pi_1(E^0,\Delta_S^0) \eqqcolon \pi_1(E^0,\Delta^0).
\]
Indeed, it suffices to prove the claim for a flip or a divisorial contraction.
First, assume that $E_i\dashrightarrow E_{i+1}$ is a flip.
Given that this is a surgery in codimension at least two, it follows that
\[
\pi_1(E^0_{i+1},\Delta_{E_{i+1},s}^0) \simeq 
\pi_1(E^0_{i},\Delta_{E_i,s}^0).
\]
Indeed, cutting out a subset of codimension at least two from a log smooth orbifold does not change its orbifold fundamental group.
Now, assume that
$\pi\colon E_i\rightarrow E_{i+1}$ is a divisorial contraction of a prime divisor $P$.
If the image of $P$ lies in the singular locus of $(E_{i+1},\Delta_{E_{i+1},s})$, then we have that
\[
\pi_1(E_{i+1}^0,\Delta^0_{E_{i+1},s})\simeq
\pi_1(E_i^0\setminus \Supp(P),\Delta^0_{E_{i},s})) \rightarrow 
\pi_1(E^0_i,\Delta^0_{E_{i},s}),
\]
and the second homomorphism is surjective.
If the image of $P$ does not lie in the singular locus
of $(E_{i+1},\Delta_{E_{i+1},s})$,
we still have that $\pi(P)$ has codimension at most two in $E_{i+1}$, hence 
\[
\pi_1(E_{i+1}^0,\Delta^0_{E_{i+1},s}) 
\simeq 
\pi_1(E_{i+1}^0\setminus \pi(P),\Delta^0_{E_{i+1},s})
\simeq 
\pi_1(E_i^0\setminus \Supp(P),\Delta^0_{E_{i},s})) \rightarrow 
\pi_1(E^0_i,\Delta^0_{E_{i},s}),
\]
and the last homomorphism is surjective.
We conclude that there is a natural surjection
\[
\pi_1(E^0_k,\Delta^0_{E_k,s})\rightarrow
\pi_1(E^0,\Delta^0).
\]
Thus, $\pi_1(E^0,\Delta^0)$ is finite.
Note that the Jordan property of rank $n$ behaves well under surjections.
Then, the second statement follows from the Proposition~\ref{prop:Jordan-global-fundamental} applied to the group
$\pi_1(E^0_k,\Delta^0_{E_k,s})$.
\end{proof}

\section{Whitney stratifications and plt blow-ups}
\label{sec:whitney}

In this section, we will use the theory of Whitney stratifications
to prove that, to a plt blow-up of a $n$-dimensional klt singularity $x \in X$, we can associate a loop that lies in the center subgroup $Z(\pi^{\rm reg}_1(X,x))$ of the regional fundamental group $\pi^{\rm reg}_1(X,x)$. 
We refer to Remark~\ref{rem:loop-plt} for a systematic definition of the loop associated to a plt blow-up.

We start with some basic properties of stratification theory.
The proofs can be found in~\cites{Gor81,GM83,GM88}.
The following is the definition of Whitney stratification.

\begin{definition}
{\em Let $X$ be an analytic space that can be embedded into a smooth manifold $M$.
A {\em Whitney stratification} for the space $X$ is a chain of inclusions
\[
X_0\subset X_1 \subset X_2\subset \dots \subset X_m=X,
\]
so that $X_i$ is open in $X$ and $W_i\coloneqq X_i\setminus X_{i-1}$ is a smooth complex manifold 
which is closed in $X_i$.
Moreover, for each $i$, we have a tubular neighborhood $U_i$ of $W_i$ in $M$,
with a {\em tubular distance function} $\rho_i\colon U_i\rightarrow \rr$, 
and a {\em tubular projection} $\pi_i \colon U_i \rightarrow W_i$.
This tubular projection is a topological retraction.
The complex manifolds $W_i$ are the {\em strata} of the Whitney stratification.

We define $U_i(\epsilon)$ to be all the points of $U_i$ whose distance function 
is strictly less than $\epsilon$, 
and by $S_i(\epsilon)$ to be all the points of $U_i$ whose distance function
equals $\epsilon$.
Restricting $\pi_i$ to $S_i(\epsilon)$, we have a natural sphere-bundle structure on $W_i$.
}
\end{definition}

\begin{lemma}\label{lem:Whitney's-condition-B}
Let $X$ be an analytic space that can be embedded into a smooth manifold.
There exists a Whitney stratification satisfying the following conditions. 
\begin{enumerate}
\item The map $(\pi_i, \rho_i)\colon U_i(\epsilon) \rightarrow W_i\times [0,\epsilon)$ has surjective
differential when restricted to any stratum $W_j \subset \overline{W_i}$.
\item For each $W_j\subset  \overline{W_i}$, the relations
$\pi_j \circ \pi_i = \pi_j$, and $\rho_i \circ \pi_j = \rho_i$
hold whenever both sides of the equality are well-defined.
\item For all $\epsilon$ sufficiently small, the sphere bundles $S_i(\epsilon)$ are multi-transverse,
i.e., for any collection of strata $W_{a_1},\dots, W_{a_k}$ and 
any disjoint collection of strata $W_{b_1},\dots, W_{b_j}$, the intersection
$S_{a_1}(\epsilon)\cap \dots \cap S_{a_k}(\epsilon)$ is transverse to the intersection
$S_{b_1}(\epsilon)\cap \dots \cap S_{b_j}(\epsilon)$,
and is also transverse to any other stratum.
\end{enumerate}
\end{lemma}

The existence of a stratification with such conditions is proved in~\cite{Gor81}.
The objects that admit a stratification as above are called {\em Whitney objects}.
Given a Whitney object $X$, we will define an important deformation retraction 
$f\colon U(\epsilon)\rightarrow X$.
The following lemma is proved in~\cite{Gor81}*{\S7}.

\begin{lemma}\label{lem:stratum-preserving-homeomorphisms}
Let $X$ be a Whitney object.
Fix $\epsilon$ small enough
satisfying the condition $(3)$ of Lemma~\ref{lem:Whitney's-condition-B}.
We can find stratum preserving retractions $r_i \colon U_i(2\epsilon)-W_i \rightarrow S_i(2\epsilon)$
for each stratum $W_i$ of $X$, such that whenever $W_j\subset \overline{W_i}$, the following conditions hold:
$r_i|_{W_j}$ is smooth, 
$r_j \circ r_i = r_i \circ r_j$,
$\rho_j \circ r_i =\rho_j$,
$\rho_i \circ r_j=\rho_i$,
$\pi_i \circ r_i = \pi_i$, and 
$\pi_j \circ r_i = \pi_j$.
In particular, we can define stratum preserving homeomorphisms 
\[
h_i \colon U_i(2\epsilon)-W_i \rightarrow S_i(2\epsilon) \times (0,2\epsilon),
\]
defined by $h_i(p)=(r_i(p),\rho_i(p))$.
\end{lemma}

\begin{corollary}\label{cor:extension-h_i}
Each homeomorphism $h_i$ extends to a homeomorphism between $U_i(2\epsilon)$ and the mapping cylinder of $\pi_i |_{S_i(2\epsilon)}$.
\end{corollary}

An important deformation retract $f\colon U(\epsilon)\rightarrow X$ is defined in~\cite{Gor81}*{7.1}.
We will recall its construction by using the homeomorphisms $h_i$.
We define $U(\epsilon)=\bigcup U_i(\epsilon)$, where the union runs over all the strata of the Whitney object $X$.
Fix a smooth non-decreasing function $q\colon \rr\rightarrow\rr$, so that
$q(t)=0$ for $t \leq \epsilon$, and $q(t)=t$ for each $t\geq 2\epsilon$.
For each stratum $W_i$ of the Whitney object $X$, we may define the function
\[
H_i(p)\coloneqq
\left\{
	\begin{array}{ll}
		p & \mbox{if } p\not\in U_i(2\epsilon) \\
		h_i^{-1}(r_i(p),q(\rho_i(p))) & \mbox{if }p\in U_i(2\epsilon)
	\end{array}
\right.
\]
The function $H_i$ is continuous, homotopic to the identity, and if $p\in U_i(\epsilon)$, then $H_i(p)=\pi_i(p)$.
We define $\tilde{f}\colon U_i(2\epsilon)\rightarrow U(2\epsilon)$
to be the composition $H_{W_1}\circ H_{W_2} \circ \dots \circ H_{W_m}$,
where the composition runs over all the strata of $X$ in any order.
Then, $f$ is just the restriction of $\tilde{f}$ to $U(\epsilon)$.
It is proved in~\cite{Gor81}*{7.2} that $f$ is a deformation retract.

\begin{definition}{\em
Let $X$ be a Whitney object and $W$ a stratum.
Fix $\epsilon > 0$.
We define the {\em $\epsilon$-interior of $W$}, denoted by $W^{\rm int}_\epsilon$, to be
\[
W^{\rm int}_\epsilon = W - \bigcup_j \rho_j^{-1}[0,\epsilon),
\]
where the union is taken over all the strata $W_j$ satisfying $W_j \subset \overline{W}$.
The {\em interior} $W^{\rm int}$ of the stratum is just the $0$-interior.
To avoid confusion, we will often say that this is the interior with respect to the Whitney stratification.
Given a Whitney object $X$, we may write $X_0=W_0$,
where $W_0$ is the unique stratum which has codimension zero on $X$.
Hence, in this case we have that $X_0=X_0^{\rm int}$.
We may just write $X^{\rm int}_\epsilon$ to refer to $W^{\rm int}_\epsilon$, when the Whitney stratification 
is clear from the context.
Furthermore, if $\epsilon$  is zero, we may suppress it from the notation.
}
\end{definition}

We prove a lemma regarding the restriction of the retraction $f$ to the
interior of the largest stratum in the Whitney stratification.

\begin{lemma}\label{lem:trivial-bundle-2e-interior}
Let $X$ be a Whitney object immersed in $M$.
Let $X_0$ be the unique stratum of codimension zero. 
Assume that the tangent bundle of the interior of the stratum $X_0$ in $M$ is trivial.
Then, the restriction of the retraction 
$f\colon U(\epsilon)\rightarrow X$ to $X^{\rm int}_{0,2\epsilon}$
is a trivial bundle.
\end{lemma}

\begin{proof}
Recall that we denote by $X^{\rm int}_{2\epsilon}$ the 
$2\epsilon$-interior of the codimension zero stratum $X_0$ of
the Whitney stratification of $X$.
Note that the restriction of $f$ to $X^{\rm int}_{2\epsilon}$ is just 
\[
H_0\coloneqq
H_{X_0}|_{X^{\rm int}_{2\epsilon}}.
\]
Indeed, all the other functions $H_i$ will be trivial on the analytic open set
$X^{\rm int}_{2\epsilon}$.
Hence, it suffices to prove that $H_{X_0}$ restricted to $X^{\rm int}_{2\epsilon}$ gives a trivial bundle.
By assumption, every point $p\in U(\epsilon)$ belongs to $U_{X_0}(\epsilon)$.
Hence, on $f^{-1}(X^{\rm int}_{2\epsilon})$ the continuous function $H_{X_0}$ is just $\pi_{X_0}$.
Moreover, on $X^{\rm int}_{2\epsilon}$ the tangent bundle is trivial,
hence $\pi_{X_0}$ is a trivial bundle.
\end{proof}

\begin{lemma}\label{lem:surjection-fundamental-groups-interior}
Let $X$ be a Whitney object.
Let $\gamma_1,\dots, \gamma_k$ be finitely many elements of  $\pi_1(X^{\rm int}_0)$
and
\[
\mathcal{R}_1(\gamma_1,\dots, \gamma_k),
\dots,
\mathcal{R}_j(\gamma_1,\dots,\gamma_k)
\]
be finitely many relations which hold on $\pi_1(X^{\rm int}_0)$.
Let $\phi \colon \pi_1(X^{\rm int}_\delta)\rightarrow \pi_1(X^{\rm int}_0)$
be the natural homomorphism induced by the inclusion $X^{\rm int}_\delta \subset X^{\rm int}_0$.
Then, for $\delta>0$ small enough, we may find 
$\gamma'_1,\dots, \gamma'_k \in \pi_1(X^{\rm int}_\delta)$
so that 
$\phi(\gamma'_i)=\gamma_i$ for each $i\in \{1,\dots, k\}$,
and the relations
\[
\mathcal{R}_1(\gamma'_1,\dots,\gamma'_k),\dots,
\mathcal{R}_j(\gamma'_1,\dots, \gamma'_k) 
\]
hold on $\pi_1(X^{\rm int}_\delta)$.
\end{lemma}

\begin{proof}
Let $W_1,\dots, W_m$ be all the strata of $X$ that are not equal to $X_0$.
Then $X^{\rm int}_0 = X - \bigcup_{i=1}^m W_i$.
Let $\gamma$ be an element of $\pi_1(X^{\rm int}_0)$.
The element $\gamma$ can be represented by a loop in $X$ that is disjoint from $\bigcup_{i=1}^m W_i$.
Observe that both sets $\gamma$ and $\bigcup_{i=1}^m W_i$ are closed in $X$.
Hence, we may find $\delta$ small enough so that
$\gamma$ and $\bigcup_{i=1}^m \rho^{-1}_i[0,\delta)$ are disjoint.
In particular, $\gamma$ is contained in $\pi_1(X^{\rm int}_\delta)$.
Hence, the natural map $\pi_1(X^{\rm int}_\delta)\rightarrow \pi_1(X^{\rm int}_0)$ contains $\gamma$ in its image.
Then, it suffices to shrink $\delta$ enough, so the above argument works for a set of genenerators of a given $k$-tuple $\gamma_1,\dots, \gamma_k$.

On the other hand, any relation $\mathcal{R}(\gamma_1,\dots, \gamma_k)$
can be represented as a homotopy of loops in $X$,
so that the homotopy is disjoint from $\bigcup_{i=1}^m W_i$.
Since the image of the homotopy is closed in the analytic topology
of $X\setminus \bigcup_{i=1}^m W_i$, we may choose $\delta>0$ small enough so that
such homotopy can be realized in $X^{\rm int}_\delta$.
In particular, the relation $\mathcal{R}(\gamma'_1,\dots,\gamma'_k)$ still holds in $\pi_1(X^{\rm int}_\delta)$.
\end{proof}

\begin{corollary}
Let $X$ be a Whitney object such that $\pi_1(X_0^{\rm int})$ is finitely presented. Then for $\delta>0$ small enough, the natural homomorphism 
$\phi \colon \pi_1(X^{\rm int}_\delta)\rightarrow \pi_1(X^{\rm int}_0)$ is an isomorphism.
\end{corollary}

Let $x\in (X,\Delta)$ be a klt singularity, and $\pi \colon Y \rightarrow X$ be a plt blow-up of $x\in (X,\Delta)$.
We denote by $E$ the exceptional divisor of $\pi$.
Applying the theory of Whitney stratifications to $E$, there exists a retraction $f\colon U(\epsilon)\rightarrow E$, see~\cite{TX17}*{\S3}.
Moreover, we may assume that $\mathcal{O}(E)$ is trivial on each stratum of $E$, see~\cite{TX17}*{\S3.1}.
The following is an enhancement of a proposition proved in~\cite{TX17}.

\begin{proposition}\label{prop:explicit-v_0}
Let $x \in (X,\Delta)$ be a klt singularity, and $\pi \colon Y \rightarrow X$ be a plt blow-up of $x\in X$ with exceptional divisor $E$.
Then, there exist
\begin{itemize}
\item a big open subset $E^0$ of $E$ on which $(E,\Delta_E)$ is log smooth; and
\item a standard boundary $B_E \leq \Delta_E$ so that
$(E,B_E)$ is a weakly log Fano type pair,
\end{itemize}
satisfying the following conditions.
Let $V^0=f^{-1}(E^0)\setminus E^0$, 
$\Delta_V^0$ the restriction of $\Delta_Y$ to $V^0$, 
and $(E^0,B^0)$ the restriction of $(E,B_E)$ to $E^0$.
Then, 
\begin{enumerate}
\item we have a differentiable fiber bundle morphism which is homotopic to $f$
\[
0
\rightarrow
\mathbb{D}^0
\rightarrow
(V^0,\Delta_{V,s}^0)
\rightarrow 
(E^0,B^0)
\rightarrow 
0,
\]
where $\mathbb{D}^0$ is a punctured disk,
and both $(V^0,\Delta_{V,s}^0)$ and $(E^0,B^0)$ are
regarded as orbifolds;
\item there is a short exact sequence of groups
\[
1 \rightarrow
\mathbb{Z}/m\mathbb{Z}
\rightarrow 
\pi_1(V^0,\Delta_{V,s}^0)
\rightarrow 
\pi_1(E^0,B^0)
\rightarrow 
1;
\]
\item 
the inclusion $V^0 \subset X$ induces a surjective homomorphism 
$\pi_1(V^0,\Delta_V^0)\rightarrow \pi_1^{\rm reg}(X,\Delta,x)$.
\end{enumerate}
\end{proposition}

In the following remark, we give an interpretation of the exact sequence of Proposition~\ref{prop:explicit-v_0}(2).

\begin{remark}\label{rem:exact-seq}
{\em
Consider a klt singularity $x\in (X,\Delta)$ and a plt blow-up $\pi\colon Y\rightarrow X$.
We can regard $X$ as a small analytic neighborhood of $x$.
Let $X'\rightarrow X$ be the closure of the cover of $X\setminus {\rm Sing}(X)$ defined by $\pi_1^{\rm reg}(X,\Delta,x)$. Let $Y'$ be the normalization of the dominant component of the fiber product.
Then, $\pi_1^{\rm reg}(X,\Delta,x)$ acts on $Y'$ and its quotient by this action is $Y$.
Furthermore, $Y'\rightarrow X'$ is a plt blow-up and $\pi_1^{\rm reg}(X,\Delta,x)$ acts on the (irreducible) exceptional divisor $E'$.
The group $\zz/m\zz$ is the largest abelian subgroup of $\pi_1^{\rm reg}(X,\Delta,x)$ which acts trivially on $E'$.
}
\end{remark}

The following remark shows that we can associate a class in $\pi_1^{\rm reg}(X,\Delta,x)$ to each plt blow-up of the singularity. 

\begin{remark}\label{rem:loop-plt}
{\em 
By Proposition~\ref{prop:explicit-v_0}, given a klt singularity $x\in (X,\Delta)$ and a plt blow-up $\pi \colon Y\rightarrow X$ with exceptional divisor $E$, we can naturally define a class $\gamma_\pi \in \pi_1^{\rm reg}(X,\Delta,x)$ of finite order.
The class $\gamma_\pi$ is represented by any loop $\bar{\gamma}_\pi \subset (V^0,\Delta_V^0)$ whose class generates the fundamental group $\pi_1(\mathbb{D}^0)$.
Here, $\mathbb{D}^0$ denotes the fiber in the differential fiber bundle structure as in Proposition~\ref{prop:explicit-v_0} (1).
In particular, $[\bar{\gamma}_\pi] \in \pi_1(V^0)$ is a generator of the subgroup $\mathbb{Z}/m\mathbb{Z} \subset \pi_1(V^0,\Delta_V^0)$ given in Proposition~\ref{prop:explicit-v_0} (2).
As the homomorphism $\pi_1(V^0,\Delta_V^0) \to \pi_1^{\rm reg}(X,\Delta,x)$ is surjective, Proposition~\ref{prop:explicit-v_0} (3) immediately shows that $\gamma_\pi=[\bar{\gamma}_\pi] \in \pi_1^{\rm reg}(X,\Delta,x)$ is of finite order.
\\
We will refer to any loop $\gamma_\pi$ as {\em the loop around $E$}; in order to simplify the notation, in the remainder of the paper, we will denote by $\gamma_\pi$ not only the class constructed above inside $\pi_1^{\rm reg}(X,\Delta,x)$, but also any loop in said class.
}
\end{remark}

\begin{proof}
In the case that $\Delta=0$, 
the three statements follow from~\cite{TX17}*{\S 3}
and~\cite{Bra20}*{\S~11, \S~12}.
In the case that $\Delta$ is non-trivial, 
we explain how to define the standard boundary $B^0$ on $E^0$.
Let $P$ be a prime divisor on $E^0$.
We denote by $n_P$ the positive integer so that $P$ appears with coefficient $1-\frac{1}{n_P}$ in the
different of $(K_Y+E)|_E$.
We denote by $m_P$ the largest integer so that some prime
component of $\Delta^0_V$ with coefficient greater than or equal to $1-\frac{1}{m_P}$ contains $P$ in its support.
Note that, if no prime component of $\Delta^0_V$ contains $P$ in its support, then $m_P=1$.
Analogously, if $E$ is Cartier at the generic point of $P$, then we have that $n_P=1$.
We set $B^0 \coloneqq \sum_{P\text{ prime}}\left( 1-\frac{1}{m_Pn_P}\right)P$.
The above sum is finite, 
and we have that
\[
K_{E^0}+D^0 \leq K_{E^0}+B^0 \leq 
K_{E^0}+\Delta^0_E,
\]
by the adjunction formula for plt pairs (see, e.g.~\cite{Sho92}*{3.2}).
Here, $D_E$ is the different obtained by adjunction
$(K_Y+E)|_{E^0}=K_{E^0}+D^0$.
Similarly, we have $(K_Y+\pi^{-1}_*(\Delta)+E)|_{E^0}=K_{E^0}+\Delta^0_E$.
Indeed, note that 
\[
D^0 =\sum_{P\text{ prime}} \left(
1-\frac{1}{n_P}\right) P, \text{ and }
\]
\[
\Delta^0_E=\sum_{P \text { prime}} 
\left( 
1-\frac{1}{n_P} 
+ \sum \frac{{\rm
mult}_Q(\Delta_Y){\rm coeff}_Q(\Delta_Y)}{n_P} 
\right) 
\geq 
\sum_{P\text{ prime}} 
\left( 1-\frac{1}{n_P}+\frac{1-\frac{1}{m_P}}{n_P} \right)= B^0,
\]
where the sum runs over all prime divisors $Q$ that contain $P$ in their support.
Note that for each prime divisor $P$ on $E^0$ with $m_P>1$, there is a unique analytic prime component of $\Delta_{Y,s}\setminus E$ containing $P$ in its support.
Indeed, note that all components of $\Delta_{Y,s}$ have coefficient at least $\frac{1}{2}$.
Furthermore, since the complexity of a log canonical pair is non-negative (see, e.g.~\cite{Kol92}*{18.22}), we conclude that either $\Delta_{Y,s}$ has exactly two analytic components with coefficient $\frac{1}{2}$ through $P$, or a unique one.
In the former case, the singularity $(Y,\Delta_{Y,s})$ is log canonical but not plt at the generic point of $P$, leading to a contradiction.
Furthermore, such component must have multiplicity one at $P$.
Hence, the retraction $f$ induces a morphism of
orbifolds
$(V^0,\Delta_{V,s}^0)\rightarrow (E^0,B^0)$,
which is locally analytically an $m_P$-fold quotient
of the usual fiber bundle structure
$V^0\rightarrow (E^0,D^0)$ (see, e.g.,~\cite{TX17}*{\S 3}).
Thus, $(V^0,\Delta_{V,s}^0)\rightarrow (E^0,B^0)$
carries a punctured disk bundle struture as well.

Part (2) follows from the first part by taking 
the long exact sequence in homotopy groups (see~\cite{TX17}*{Remark 3.3}).
We only need to argue that $\pi_1(V^0,\Delta^0_{V,s})$ is finite.
This follows from the fact that $\pi_1(V^0)$ is finite
and the proof of~\cite{TX17}*{Lemma 3.5}.
Note that $(V^0,\Delta_{V,s}^0)$ can be realized as an analytic log smooth open subset of the regional log smooth locus of a $\qq$-factorial pair with standard coefficients.
Indeed, $(V^0,\Delta_{V,s}^0)$ is an open analytic subset of the log smooth locus of a klt singularity, and a small $\qq$-factorialization induces an isomorphism on it.

Part (3) follows from the proof given in~\cite{Bra20}*{\S 12}.
\end{proof}

Now, we are ready to prove the main theorem of this section.
We prove that the cycles of the local fundamental group of a $n$-dimensional klt singularity corresponding to plt blow-ups lie in the center subgroup.

\begin{theorem}\label{thm:plt-loop-commutes}
Let $x\in (X,\Delta)$ be a $n$-dimensional klt singularity and $\pi \colon Y\rightarrow X$
be a plt blow-up.
Consider $V^0\subset Y$ as in Proposition~\ref{prop:explicit-v_0}.
Then, $\gamma_\pi \in Z(\pi_1(V^0,\Delta^0_{V,s}))$.
\end{theorem}

\begin{proof}
Recall that we may assume that the Whitney stratification of $E$
is so that $\mathcal{O}(E)$ is trivial on each stratum,
i.e., the line bundle $\mathcal{O}_Y(E)|_E$ trivializes on each stratum.
In particular, $\mathcal{O}(E)$ is the trivial line bundle when restricted
to $E^{\rm int}_0$, i.e., $\mathcal{O}(E)$ is trivial on the interior of $E$
with respect to this Whitney stratification.
Without loss of generality, we may assume that 
$E^{\rm int}_0 \cap \Delta=\emptyset$.
This may be achieved by adding to $\Delta$ prime components with coefficient zero.
Denote by $E^0$ a big open subset of $E$ on which $(E,\Delta)$ is log smooth.
Again, we may assume that $E^{\rm int}_0 \subset E^0$.
Let $B^0$ as in Proposition~\ref{prop:explicit-v_0}.
By Theorem~\ref{thm:finiteness-weakly-LF}, we know that $\pi_1(E^0,B^0)$ is finite.
We know that $\pi_1(E^{\rm int}_0) \rightarrow \pi_1(E^0,B^0)$ is a surjective homomorphism (e.g, see~\cite{CLS11 }*{Theorem 12.1.5}).
Hence, by Lemma~\ref{lem:surjection-fundamental-groups-interior}, we may find $\delta$ small enough so that
the homomorphism induced by the inclusion
\[
\pi_1(E^{\rm int}_\delta ) \rightarrow \pi_1(E^{\rm int}_0),
\] 
surjects onto a subgroup whose image in $\pi_1(E^0,B^0)$ generates the whole group.
In particular, the homomorphism
\[
\pi_1(E^{\rm int}_\delta) \rightarrow \pi_1(E^0,B^0)
\]
induced by the natural inclusion is surjective.
Observe that the above surjectivity is preserved if we shrink $\delta$.
On the other hand, we may choose $\epsilon$ small enough, 
so that $\delta > 2\epsilon$, and by Lemma~\ref{lem:surjection-fundamental-groups-interior}, 
the restriction of $f\colon U(\epsilon)\rightarrow E$ to $E^{\rm int}_{2\epsilon}$ is a trivial bundle.
The exact sequence
\begin{align*}
\xymatrix{
1 \ar[r] &
\mathbb{Z}/m\mathbb{Z} \ar[r] &
\pi_1(V^0,\Delta_{V,s}^0) \ar[r] &
\pi_1(E^0,B^0)\ar[r] & 
1,
}
\end{align*}
as in Proposition~\ref{prop:explicit-v_0}, is just the exact sequence induced by the long exact sequence of homotopy groups of the differentiable bundle $\mathbb{D}^0\rightarrow (V^0,\Delta_{V,s}^0) \rightarrow (E^0,B^0)$.
On the other hand, such differentiable bundle structure trivializes on $E^{\rm int}_{2\epsilon}$,
hence we get the following commutative diagram
\begin{equation}\label{diagram-lifting}
 \xymatrix{
1\ar[r]\ar[d] & 
\mathbb{Z} \ar[r]\ar[d] & 
\zz \times \pi_1(E^{\rm int}_{2\epsilon}) \ar[r]\ar[d] & 
\pi_1(E^{\rm int}_{2\epsilon})\ar[r]\ar[d]  & 
1 \ar[d] \\
1\ar[r] & 
\mathbb{Z}/m\mathbb{Z} \ar[r] &  
\pi_1(V^0,\Delta_{V,s}^0) \ar[r] & 
\pi_1(E^0,B^0) \ar[r] & 
1 
}
\end{equation}
The morphism $\zz \rightarrow \mathbb{Z}/m\mathbb{Z}$ in~\eqref{diagram-lifting} is surjective: in fact, both groups are generated by a non-trivial loop on the general
fiber of the differentiable bundle.
Hence, by the Five Lemma, we conclude that there is a surjection $\zz \times \pi_1(E^{\rm int}_{2\epsilon}) \rightarrow \pi_1(V^0,\Delta_{V,s}^0)$.
We denote by $\gamma'_\pi$ a generator of $\mathbb{Z}$ in~\eqref{diagram-lifting}.
We will use the same notation also for the image of $\gamma'_
\pi$ in $\zz \times \pi_1(E^{\rm int}_{2\epsilon})$; hence, $\gamma'_
\pi \in \zz \times \pi_1(E^{\rm int}_{2\epsilon})$ generates $\zz\times \{ e\}$, where $e$ denotes the identity of $\pi_1(E^{\rm int}_{2\epsilon})$.
Thus, $\gamma'_\pi \in Z(\zz \times \pi_1(E^{\rm int}_{2\epsilon}))$ as $\zz\times \{ e\} \subset Z(\zz \times \pi_1(E^{\rm int}_{2\epsilon}))$.
Thus, as $\gamma_\pi \in \pi(V_0,\Delta_{V,s}^0)$ is the image of $\gamma'_\pi$ via the surjective morphism in~\eqref{diagram-lifting}, we conclude that $\langle \gamma_\pi \rangle \subset Z(\pi_1(V^0,\Delta_{V,s}^0))$.
\end{proof}

\begin{corollary}
Let $x\in (X,\Delta)$ be a $n$-dimensional klt singularity.
Then, for every plt blow-up $\pi \colon Y \rightarrow X$ of $x\in (X,\Delta)$,  $\gamma_\pi \in Z(\pi_1^{\rm reg}(X,\Delta,x))$.
\end{corollary}

\begin{proof}
By Proposition~\ref{prop:explicit-v_0}, we know that the homomorphism $\pi_1(V^0,\Delta_{V,s}^0)\rightarrow \pi_1^{\rm reg}(X,\Delta,x)$ 
is surjective.
On the other hand, by Theorem~\ref{thm:plt-loop-commutes}, $\gamma_\pi \in Z(\pi_1(V^0,\Delta_{V,s}^0))$; 
hence, by Proposition~\ref{prop:explicit-v_0}(3), also $\gamma_\pi \in Z(\pi_1^{\rm reg}(X,\Delta,x))$.
\end{proof}

\section{Proof of the Jordan property}\label{sec:proof-jordan}

In this section, we prove the main theorem of this article.

\begin{proof}[Proof of Theorem~\ref{intro-theorem:Jordan-klt}]
Let $x\in (X,\Delta)$ be a $n$-dimensional klt singularity.
Let $\pi \colon Y \rightarrow (X,\Delta)$ be a plt blow-up centered at $x\in (X,\Delta)$.
By Proposition~\ref{prop:explicit-v_0}, we have a differentiable fiber bundle morphism 
\[
0\rightarrow \mathbb{D}^0 \rightarrow
(V^0,\Delta^0_{V,s})\rightarrow (E^0,B^0)\rightarrow 0,
\]
a short exact sequence of groups
\[
1\rightarrow \mathbb{Z}/m\mathbb{Z}\rightarrow \pi_1(V^0,\Delta^0_{V,s})\rightarrow \pi_1(E^0,B^0)\rightarrow 1,
\]
and a surjection $\pi_1(V^0,\Delta_{V,s}^0)\rightarrow \pi_1^{\rm reg}(X,\Delta_s,x)$.
By definition of $\pi_1(V^0,\Delta_V^0)$ and $\pi_1^{\rm reg}(X,\Delta,x)$, this gives a surjection
$\pi_1(V,\Delta_V^0)\rightarrow \pi_1^{\rm reg}(X,\Delta,x)$.
Hence, in order to obtain the Jordan property with rank $n$ for the regional fundamental group
$\pi_1^{\rm reg}(X,\Delta,x)$,
it suffices to prove the Jordan property for rank $n$ for the fundamental group
$\pi_1(V^0,\Delta_V^0)$.
As in the proof of~\cite{Bra20}*{Theorem 7}, we can find a log resolution
$(Z,B_Z)$ of $(E,B)$ that is an isomorphism
over $(E^0,B^0)$, so that the inclusion $E^0\hookrightarrow Z$ induces an isomorphism
\[
\pi_1(E^0,B^0)\simeq \pi_1(Z,B_Z).
\]
We can realize the log resolution $(Z,B_Z)$ of $(E,B)$ as an embedded resolution in $Y$
that is an isomorphism over $V^0$.
We denote by $Y_Z$ the quasi-projective variety obtained by the embedded resolution of $E$ inside $Y$.\\

\noindent 
\textbf{Claim 1:} 
There exists a constant $f(n)$, only depending on $n$, satisfying the following.
There is a closed point $z\in Z$ so that the image of the homomorphism
$\pi_1^{\rm reg}(Z,B_Z,z)\rightarrow \pi_1(Z,B_Z)$ has index at most $f(n)$.

\begin{proof}[Proof of Claim 1]
Define $G \coloneqq \pi_1(E^0,B^0)\simeq \pi_1(Z,B_Z)$.
Let $(E_1^0,B_1^0)$ be the orbifold universal cover of $(E^0,B^0)$ and
$(E_1,B_1)$ be a $G$-invariant codimension two compactification as considered in the proof of Proposition~\ref{prop:Jordan-global-fundamental}.
Let $Z_1$ be the normalization of the main component of $E_1\times_{E} Z$.
Note that $Z_1$ is endowed with a $G$-action 
and a boundary $B_{Z_1}$ so that $(Z,B_Z)$
is the quotient of $(Z_1,B_{Z_1})$ by $G$.
Thus, we obtain a diagram as follows
\begin{equation}\label{quotient-G}
 \xymatrix{
 (Z_1,B_{Z_1}) \ar[r]^-{/G}\ar[d] & (Z,B_Z) \ar[d] \\
 (E_1,B_1) \ar[r]^-{/G}&  (E,B).
}
\end{equation}
Note that $(E_1,B_1)$ is weakly log Fano, hence it is rationally connected.
Since both vertical maps are birational, 
we conclude that $Z_1$ is rationally connected as well.
By~\cite{PS16}*{Theorem 4.2}, we conclude that there is a subgroup $H\leqslant G$ of index at most $f(n)$, acting with a fixed point on $Z_1$.
Here, $f(n)$ is a constant only depending on $n$.
Replacing $H$ with its normal core in $G$, we may assume that it is normal.
We show that the image $z$ of such point on $Z$ is a point for which the image of
$\pi_1^{\rm reg}(Z,B_Z,z)\rightarrow \pi_1(Z,B_Z)$ has index at most $f(n)$.
Indeed, we have a commutative diagram

\begin{equation}\label{quotient-A}
 \xymatrix{
 (U_1,B_{U_1})\ar[r]^-{/H}
 \ar[d] &
 (U_2,B_{U_2})\ar[r]^-{/(G/H)}
\ar[d]
 &
 (U,B_Z)\ar[d] \\\
 (Z_1,B_{Z_1}) \ar[r]^-{/H} & 
 (Z_2,B_{Z_2})\
 \ar[r]^-{/(G/H)}& 
(Z,B_Z) 
}
\end{equation}
so that $(U,B_Z)$ is a open analytic subset computing $\pi_1^{\rm reg}(Z,B_Z,z)$, i.e., 
$U$ is an analytic neighborhood of $z$ so that
\[
\pi_1^{\rm reg}(Z,B_Z,z)\simeq 
\pi_1(U\cap Z^{\rm reg},B_Z).
\]
Here, $U_1$ and $U_2$ are the preimages image of $U$ with respect to the quotient morphism $Z_1\rightarrow Z$.
Note that $(Z_1,B_{Z_1})$  is the universal cover of $(Z_2,B_{Z_2})$.
Indeed, over $E^0$ the quotient morphism $E_1\rightarrow E$ is the universal cover of $(E^0,B^0)$, the birational morphism $Z\rightarrow E$ is an isomorphism over $E^0$, and
$(E^0,B^0)\hookrightarrow (Z,B_Z)$ induces an isomorphism of fundamental groups.
Note that $U_2$ may have several connected components.
Let $z_2$ be a pre-image of $z$ on $U_2$.
This give us a commutative diagram as follows
\[
\xymatrix{
H \ar@{^{(}->}[r] \ar[d]_-{\id_H}
&
\pi_1^{\rm reg}(U_2,B_{U_2},z_2)
\ar[r]\ar[d] 
&
\pi_1^{\rm reg}(Z,B_Z,z)
\ar[d] \\
H \ar[r]^-{\sim} &
\pi_1(Z_2,B_{Z_2})
\ar@{^{(}->}[r]
&
\pi_1(Z,B_Z).
}
\]
Concluding the proof of the claim.
\end{proof}
Note that $z$ is a toroidal point of $Y_Z$.
This follows from the fact that $(Z,B_Z)$ is log smooth at $z$.
In particular, 
the regional fundamental group of $z$ in $(Z,B_Z)$ is a finite abelian group of rank at most $n-1$ (see, e.g.,~\cite{CLS11}*{12.1.10}).
Denote by $F$ the exceptional divisors of the projective birational morphism 
$Z\rightarrow E$.
We have that
\begin{equation}\label{ab-inf}
\pi_1^{\rm reg}(Z\setminus \Supp (F), B_Z, z)    
\end{equation}
is an abelian group with at most $n-1$ generators, which may be infinite.
Notice that $z$ may not belong to $Z \setminus \Supp(F)$, but the inverse limit of analytic open sets is centered at $z$.
Indeed, the above is the regional fundamental group of a log toric singularity after possibly cutting out some toric boundaries.
Hence, this fundamental group is abelian with at most $n-1$ generators by~\cite{CLS11}*{12.1.10}.
Note that the group in~\eqref{ab-inf} 
surjects onto $\pi_1^{\rm reg}(Z,B_Z,z)$.
We conclude that the image of
$\pi_1^{\rm reg}(Z\setminus \Supp (F), B_Z,z)$
in $\pi_1(Z,B_Z)$ via the homomorphism induced by the inclusion has index at most $f(n)$.
Let $U'_z$ be an analytic neighborhood of $z$ in $Z$ computing the fundamental group in~\eqref{ab-inf}.
After shrinking around $z$ if necessary, we may assume that 
$U$ and $U'_z$ coincide.
This means that
\[
\pi_1^{\rm reg}(Z\setminus \Supp (F), B_Z,z)\simeq 
\pi_1(U'_z \setminus \Supp (F), B_Z, z).
\]
Note that the fundamental group
of $(U'_z,\setminus \Supp (F), B_Z,z)$ is abelian generated by at most $n-1$ elements.
On the other hand,
$U_z \coloneqq U'_z\setminus \Supp (F)$ can be identified with an analytic open subset of $E^0$.
Indeed, $U_z$ is an analytic open subset of the log smooth locus computing the regional fundamental group of the image $e$ of $z$ in $(E,B)$.\\

\noindent 
\textbf{Claim 2:} The homomorphism
$\pi_1(U_z,B^0)\rightarrow \pi_1(E^0,B^0)$
has image of index at most $f(n)$.

\begin{proof}[Proof of Claim 2]
We have a commutative diagram
\[
\xymatrix{ 
\pi_1(U_z,B^0)\ar[r]^-{\sim }\ar[d]  & \pi_1^{\rm reg}(Z\setminus \Supp (F), B_Z,z) \ar[d]  \\
\pi_1(E^0,B^0)\ar[r]^-{\sim } & \pi_1(Z,B_Z).
}
\]
Since the image of the right vertical homomorphism has index at most $f(n)$,
we conclude that the image of the left vertical homomorphism has index at most $f(n)$ as well.
\end{proof}

Recall that $U_z$ can be identified with an analytic open subset of $E^0$.
Recall that there exists an open analytic subset $E^{\rm int}_{2\epsilon}$ of $E^0$ on which the retraction is trivial.\\

\noindent
\textbf{Claim 3:}
There exists an open analytic subset $U_{z,\delta}$ contained in $E^{\rm int}_{2\epsilon}$
so that $\pi_1(U_{z,\delta})$ is abelian of rank at most $n-1$ and 
$\pi_1(U_{z,\delta})\rightarrow \pi_1(E^0,B^0)$ has image of index at most $f(n)$.

\begin{proof}[Proof of Claim 3]
Consider the log resolution
$Z\rightarrow E$ of $(E,B)$.
Consider the natural stratification of $Z$ with respect to the log pair structure $(Z,B_Z)$, i.e,
 the strata are the complement of $B_Z$ and the intersections of components of $B_Z$.
Fix $\delta>0$.
Note that for $\epsilon$ small enough, we have that
$Z^{\rm int}_{\delta}$ is contained in the pre-image of 
$E^{\rm int}_{2\epsilon}$.
We define $U_{z,\delta}$ to be the intersection of $U_z$ with $Z^{\rm int}_\delta$.
Note that for $\delta$ small enough, the analytic open set $U_{z,\delta}$ is homotopic to a product of at most $n-1$ disks.
Hence, the fundamental group 
$\pi_1(U_{z,\delta})$ is abelian of rank at most $n-1$.
It suffices to check the second condition.
Note that $\pi_1(U_{z,\delta})\rightarrow \pi_1(U_z,B^0)$ is surjective.
Hence, the
homomorphism
$\pi_1(U_{z,\delta})\rightarrow \pi_1(E^0,B^0)$ induced by the inclusion has image of index at most $f(n)$.
\end{proof}

We can identify $U_{z,\delta}$ with its image in $E$.
Hence, we can consider it as an analytic subset of
$E^{\rm int}_{2\epsilon}$.
Recall that the retraction is trivial over $E^{\rm int}_{2\epsilon}$.
Let $V_{z,\delta}$ be the pre-image of $U_{z,\delta}$ with respect to the retraction.
Then, we have that $\pi_1(V_{z,\delta})$ is an abelian group of rank at most $n$.
Thus, we have a commutative diagram as follows 
\[
\xymatrix{ 
1\ar[r] & \zz\ar[r]\ar[d] & \pi_1(V_{z,\delta})\ar[r]\ar[d] &
\pi_1(U_{z,\delta})\ar[r]\ar[d] & 1\\
1\ar[r] & \mathbb{Z}/m\mathbb{Z}\ar[r] & \pi_1(V^0,\Delta_{V}^0)\ar[r] &
\pi_1(E^0,B^0)\ar[r] & 1}
\]
We conclude that the image of $\pi_1(V_{z,\delta})$ in $\pi_1(V^0,\Delta^0_V)$ has index at most $f(n)$.
Hence, $\pi_1^{\rm reg}(X,\Delta,x)$ has an abelian subgroup of index at most $f(n)$ and rank at most $n$.
Then, passing to the normal core of this abelian subgroup, 
we obtain a normal abelian subgroup 
of index at most $f(n)!$ and rank at most $n$.
\end{proof}

\section{Simultaneous index one covers}\label{sec:simultaneous-index}

In this section, we prove an application regarding simultaneous index one covers.

\begin{proof}[Proof of Corollary~\ref{introcor-1}]
By Theorem~\ref{intro-theorem:Jordan-klt},
there exists a subgroup $A\leqslant \pi_1^{\rm reg}(x,\Delta,x)$
that is an abelian finite group of rank at most $n$
and index bounded by $c'(n)$.
Hence, the group $\pi_1^{\rm reg}(X,\Delta,x)$ can be generated by at
most $n+c'(n)$ elements.
Setting $b(n)$ to be $n+c'(n)$, the claim follows.
\end{proof}

\begin{proof}[{Proof of Theorem \ref{thm:effective-index-one-cover}}]

Let $x\in X$ be a klt singularity.
Per conventions in this work, $x\in X$ is an algebraic singularity, but we may consider Zariski or analytic neighborhoods of it.
We construct a finite morphism $\phi \colon \tilde X\rightarrow X$ as follows:
\begin{itemize}
    \item 
in the case of an analytic neighborhood, 
we let $\phi \colon \tilde X\rightarrow X$ be the closure of the universal cover of $X\setminus {\rm Sing}(X)$;
    \item 
in the fully algebraic setting, 
we let $\phi\colon \tilde X\rightarrow X$ be the finite morphism constructed in Lemma~\ref{lem:com-index-one-cover}.
\end{itemize}
In both cases, the morphism $\phi$ satisfies the following conditions:
\begin{itemize} 
    \item $\phi$ is \'etale in codimension one;
    \item $x\in X$ has a unique pre-image $\tilde x \in \tilde X$; and
    \item every $\qq$-Cartier integral Weil divisor on $\tilde X$ is Cartier at $\tilde x$.
\end{itemize}
The rest of the proof proceeds without differences for both settings, hence, we do not distinguish between those two.

Let $\pi \colon Y \rar X$ be a plt blow-up centered at $x$.
Denote by $E$ the $\pi$-exceptional divisor.
Let $\tilde Y$ be the normalization of the main component of $Y \times_X \tilde X$.
As $\phi$ is Galois, we have $X=\tilde X/G$ for a finite gruop $G$.
Then, $G$ acts on $Y \times_X \tilde X$, and preserves the main component.
Thus, the main component has a $G$-action, and this lifts to its normalization $\tilde Y$ \cite{Spe}.
Thus, the morphism $\psi \colon \tilde Y \rar Y$ is a finite Galois morphism with Galois grup $G$.
Indeed, by \cite{Sta}*{Tag 01WL}, $\tilde X \times_X Y \rar Y$ is a finite morphism.
Then, by \cite{Sta}*{Tag 035C}, the inclusion of the main component of $\tilde X \times_X Y \rar Y$ into $\tilde X \times_X Y \rar Y$ is finite.
By \cite{Sta}*{Tag 035R}, the morphism from $\tilde Y$ to the main component of $\tilde X \times_X Y \rar Y$ is finite.
Finally, the induced morphism $\psi \colon \tilde Y \rar Y$ is finite, as it is the composition of finite morphisms \cite{Sta}*{Tag 01WK}.
Thus, we have $Y = \tilde Y / G$.

Since $\phi$ is \'etale in codimension 1, we have $\K \tilde X. = \phi^* \K X.$.
Thus, the induced morphism $\psi$ is \'etale in codimension 1 away from $E$.
Set $\tilde E \coloneqq (\psi^*(E))_{\rm red}$, where $\psi^*$ denotes the pull-back under a finite morphism.
By the Riemann--Hurwitz formula, we have $\psi^* (\K Y. + E) = \K \tilde Y. + \tilde E$.
Then, by \cite{KM98}*{Proposition 5.20}, it follows that $(\tilde Y, \tilde E)$ is plt.
Notice that $\tilde E$ is connected, since $x$ has a unique pre-image in $\tilde x$ and, by the Riemann--Hurwitz formula, the fiber of $\tilde Y \rar \tilde X$ over $\tilde x$ coincides with the non-klt locus of $(\tilde Y, \tilde E)$.
Thus, $\tilde E$ is irreducible and normal.
Then, it follows that $\rho \colon (\tilde Y, \tilde E) \rar Y$ is a plt blow-up with center $\tilde x$.

As the action of $G$ fixes $\tilde x$, it follows that $G$ admits a homomorphism $G \rar \mathrm{Aut}(\tilde E)$.
Let $H$ denote the kernel of this group homomorphism.
Since $G$ acts faithfully on $\tilde Y$,
by Corollary \ref{group_fix_divisor}, $H \simeq \mathbb{Z}/m\mathbb{Z}$ for some positive integer $m$.
Then, by \cite{Bir16b}*{Corollary 1.3}, the group $G' \coloneqq G/H$ satisfies the Jordan property.
In particular, there is a normal abelian subgroup $A' \subset G'$ such that the order of $G'/A'$ is bounded by $d(n-1)$, 
where $d(n-1)$ is as in Proposition \ref{prop:Jordan-Bir}.
Denote by $A$ the preimage of $A'$ in $G$.
By Proposition \ref{prop:Jordan-Bir}, we may assume that the rank of $A'$ is at most $n-1$.
Notice that $A$ is quasi-abelian, as it is the extension of the abelian group $A'$ by a copy of $\mathbb{Z}/l\mathbb{Z}$, with $l|m$.
Since $A$ is normal in $G$, it induces an intermediate cover $\tilde X \rar X_1 \rar X$, where both morphisms are Galois.
The Galois group corresponding to $X_1 \rar X$ is $G/A \simeq G'/A'$.
Therefore, the morphism has degree bounded by $d(n-1)$.
On the other hand, the Galois group of $\tilde X \rar X_1$ is $A$.
First, we can consider the intermediate cover $\tilde X \rar X_2 \rar X_1$ induced by $H \cap A \simeq \mathbb{Z}/l\mathbb{Z}$.
Notice that $X_2 \rar X_1$ is a cyclic cover.
As $H \cap A$ is normal in $A$, the morphism $\tilde X \rar X_2$ corresponds to the abelian group $A/A\cap H \simeq A'$.
As $A'$ is abelian, we can conclude the proof in a similar fashion as in Corollary~\ref{introcor-1}.
In particular, notice that, as $A'$ has rank at most $n-1$, there are at most $n$ cyclic covers, and $\tilde X = X_{k}$ for some $k \leq n+1$.
Finally, we can choose $b'(n)=d(n-1)$.
\end{proof}

\begin{lemma}\label{lem:com-index-one-cover}
Let $x\in X$ be a klt singularity. There exists a Galois cover $\gamma \colon \tilde X \rightarrow X$ with the following properties:
\begin{itemize} 
    \item $\gamma$ is \'etale in codimension one;
    \item $x\in X$ has a unique pre-image $\tilde x \in \tilde X$; and
    \item every $\qq$-Cartier integral Weil divisor on $\tilde X$ is Cartier at $\tilde x$.
\end{itemize}
\end{lemma} 

\begin{proof}
We construct a sequence of Galois covers $X_i\rightarrow X$ that are \'etale in codimension one.
We proceed inductively. We set $X_0=X$, $x_0=x$, and $G_0=\id_X$.
For every $i\geq 1$, we have a singularity $x_i\in X_i$, a $G_i$-action on $X_i$ so that the quotient is $X$, and $x_i$ is fixed by $G_i$.
Here $x_i$ denotes the unique pre-image of $x$ in $X_i$.
We construct $X_{i+1}$ as follows.
Let $D_i$ be a $\qq$-Cartier integral Weil divisor on $X_i$ that is not Cartier at $x_i$.
Up to shrinking around $x_i$, we may assume that the Cartier index of $D_i$ in $X_i$ is attained at $x_i$.
Let $m_i$ be the Cartier index of $D_i$ at $x_i$.
We will construct a $G_i$-equivariant index one cover of $D_i$.
Let $H$ be the subgroup on the local class group of $X_i$ at $x_i$ generated by elemenets of the form
$g \cdot D$ with $g\in G_i$.
Notice that $G_i$ is finite, the Cartier index of each $g \cdot D$ is finite, and the local class group is abelian.
Thus, $H$ is a finite group.
We define $X_{i+1}$ to be the relative spectrum of $\bigoplus_{D\in H} \mathcal{O}_X(D)$.
Then, $X_{i+1}\rightarrow X_i$ is the quotient morphism induced by the action of $H$.
Since we can decompose this cover as a sequence of cyclic covers, it follows from \cite{Kol13b}*{2.48.(1)} that
$x_i$ has a unique pre-image $x_{i+1}$ in $X_{i+1}$.
Note that the pull-back of $D_i$ to $X_{i+1}$ is a Cartier divisor.
We have a finite group $G_{i+1}$ acting on $X_{i+1}$, which fits in an exact sequence
\[
1\rightarrow H \rightarrow G_{i+1}\rightarrow G_i \rightarrow 1,
\]
so that $X_{i+1}\rightarrow X_i$ is $G_{i+1}$-equivariant, see~\cite{AG10}*{Theorem 5.1}).
The quotient of $X_{i+1}$ by $G_{i+1}$ is $X$.
Each morphism $X_i\rightarrow X_{i-1}$ is finite and \'etale in codimension one, but none of them is \'etale over $x_{i-1}$.
By~\cite{GKP}*{Theorem 1.1}, the above sequence is finite. 
Hence, for some $j\geq 1$ we have that every $\qq$-Cartier Weil divisor on $X_j$ is Cartier at $x_j$.
Thus, we can let $\tilde X=X_j$ and $\tilde x=x_j$.
\end{proof}

\section{Local class group}\label{sec:local-class}

In this section, we prove an application to the local class group of $n$-dimensional klt singularities.

Let $x \in X$ be a closed point in $X$.
We denote by $\mathcal{O}_{X,x}^{\rm hol}$ the analytic local ring of $x \in X$ and by $X_{\rm sing}$ the singular locus of $X$.

\begin{lemma} \label{lemma_injection}
Let $x \in X$ be a rational singularity.
There exists an injective morphism
\[
{\rm Cl}(\mathcal{O}_{X,x}^{\rm hol}) \hookrightarrow \varinjlim_U H^2(U \setminus X_{\rm sing}, \mathbb Z),
\]
where $U$ runs over the analytic neighborhoods of $x$.
\end{lemma}

\begin{proof}
We follow the proof of \cite{Fle}*{Satz 6.1}.
By the exactness of direct limits and the exponential sequence, there exists an exact sequence
\[
\varinjlim_U H^1(U \setminus X_{\rm sing}, \O X. ^{\rm hol}) \xrightarrow{\alpha} \varinjlim_U {\rm Pic}(U \setminus X_{\rm sing}) \xrightarrow{\beta} \varinjlim_U H^2(U \setminus X_{\rm sing}, \mathbb Z).
\]
Furthermore, there exists also a morphism 
\[
j \colon {\rm Cl}(\mathcal{O}_{X,x}^{\rm hol}) \rar \varinjlim_U {\rm Pic}(U \setminus X_{\rm sing}).
\]
which is injective, as $x \in X$ is a normal singularity: indeed, for every neighborhood $U$, by the $S_2$ property, we have $\O X.^{\rm hol}(U)=\O X.^{\rm hol}(U \setminus X_{\rm sing})$, and we have a natural inclusion $\O X.^{\rm hol}(U)\subset \O X,x.^{\rm hol}$.
\newline
Now, assume that $\beta(j(\mathfrak{a}))=0$ for some $\mathfrak{a} \in \mathrm{Cl}(\O X,x.^{\rm hol})$.
Then, there exists $z \in \varinjlim_U H^1(U \setminus X_{\rm sing}, \O X. ^{\rm hol})$ so that $\alpha(z)=j(\mathfrak a )$.
For every $n \in \mathbb N$, we consider the element $\frac{z}{n} \in \varinjlim_U H^1(U \setminus X_{\rm sing}, \O X. ^{\rm hol})$.
Therefore, the element $\frac{j(\mathfrak a)}{n}$ is well defined for every $n \in \mathbb N$.
Then, by \cite{Fle}*{Lemma 6.3}, the element $\frac{\mathfrak a}{n}$ is also well defined.
Therefore, if $\mathfrak a \neq 0$, ${\rm Cl}(\O X,x.^{\rm hol})$ has a divisible subgroup.
On the other hand, by \cite{Sto69}, the local class group $\mathrm{Cl}(\O X,x.^{\rm hol})$ of a normal analytic space is finitely generated provided that $(R^1\pi_*\O Y.)_x=0$, where $\pi \colon Y \rar X$ is a resolution of singularities.
By assumption, $x \in X$ is a rational singularity, thus $(R^1\pi_*\O Y.)_x=0$ for any resolution of singularities.
Therefore, $\mathfrak a \neq 0$ is impossible, and $\beta \circ j$ is injective.
\end{proof}

We prove an analogous result for the algebraic local class group of $x \in X$.

\begin{corollary} \label{cor_inj}
Let $x \in X$ be a rational singularity.
There exists an injective morphism
\[
{\rm Cl}(\mathcal{O}_{X,x}) \hookrightarrow \varinjlim_U H^2(U \setminus X_{\rm sing}, \mathbb Z),
\]
where $U$ runs over the analytic neighborhoods of $x$.
\end{corollary}

\begin{proof}
Consider the inclusion of rings $\O X,x. \subset \O X,x.^{\rm hol} \subset \widehat{\O X,x.}$.
By properties of the completion, the ring extensions $\O X,x. \subset \widehat{\O X,x.}$ and $\O X,x.^{\rm hol} \subset \widehat{\O X,x.}$ are faithfully flat.
Thus, by standard properties of faithful flatness, it follows that $\O X,x. \subset \O X,x.^{\rm hol}$ is faithfully flat.
Then, by \cite{Fos73}*{Corollary 6.11}, the natural morphism ${\rm Cl}(\O X,x.) \rar {\rm Cl}(\O X,x.^{\rm hol})$ is an injection.
Then, the claim follows by Lemma \ref{lemma_injection}.
\end{proof}

\begin{proof}[Proof of Corollary~\ref{thm_class_group}]
The proof is the same for both ${\rm Cl}(\O X,x.)_{\rm tor}$ and ${\rm Cl}(\O X,x.^{\rm hol})_{\rm tor}$.
So, we focus on the case of ${\rm Cl}(\O X,x.)_{\rm tor}$.
Since $(X,\Delta)$ is klt, it has rational singularities \cite{KM98}*{Theorem 5.22}.
Thus, by Corollary \ref{cor_inj}, there is an injection ${\rm Cl}(\mathcal{O}_{X,x}) \hookrightarrow \varinjlim_U H^2(U \setminus X_{\rm sing}, \mathbb Z)
$.
By the uniqueness of analytic neighborhoods \cite{Dur}, there exists an analytic neighborhood $U_0$ of $x$ so that $\varinjlim_U H^2(U \setminus X_{\rm sing}, \mathbb Z)
=H^2(U_0 \setminus X_{\rm sing}, \mathbb Z)$.
Up to shrinking $U_0$, we may also assume $\pi_1^{\rm reg}(X,x)=\pi_1(U_0 \setminus X_{\rm sing})$.
Therefore, there is an injection
\[
{\rm Cl}(\mathcal{O}_{X,x})_{\rm tor} \hookrightarrow H^2(U_0 \setminus X_{\rm sing}, \mathbb Z)_{\rm tor}.
\]
By the Universal Coefficient Theorem \cite{Hat}*{Corollary 3.3 p. 196}, we have
\[
H^2(U_0 \setminus X_{\rm sing}, \mathbb Z)_{\rm tor}=H_1(U_0 \setminus X_{\rm sing}, \mathbb Z)_{\rm tor}.
\]
By Corollary \ref{introcor-1}, $\pi_1^{\rm reg}(X,x)=\pi_1(U_0 \setminus X_{\rm sing})$ is generated by at most $b(n)$ elements, and at most $n$ of such generators have order larger than $b(n)$.
Since $H_1(U_0 \setminus X_{\rm sing}, \mathbb Z)_{\rm tor}$ is the abelianization of $\pi_1(U_0 \setminus X_{\rm sing})=\pi_1(U_0 \setminus X_{\rm sing})_{\rm tor}$, it follows that $H_1(U_0 \setminus X_{\rm sing}, \mathbb Z)_{\rm tor}$ satisfies the same property.
That is, $H_1(U_0 \setminus X_{\rm sing}, \mathbb Z)_{\rm tor}$ is generated by at most $b(n)$ elements, and at most $n$ of such generators have order larger than $b(n)$.
Since ${\rm Cl}(\O X,x.)_{\rm tor}$ is a subgroup of $H_1(U_0 \setminus X_{\rm sing}, \mathbb Z)_{\rm tor}$, the claim follows by Lemma \ref{lemma_rank}.
\end{proof}

\section{Examples and questions}\label{sec:ex}

In this final section, we collect some remarks and examples; furthermore, we propose some questions related to the results contained in this article.
We now illustrate a few interesting examples related to the klt singularities and the Jordan property.

\begin{example}\label{ex:exceptional-singularities-example}{\em
A klt singularity $x\in X$ is said to be exceptional,
if for every boundary $B\geq 0$ on $X$ so that
$(X,B)$ is a log canonical pair, $(X,B)$ is either klt or log canonical 
with a unique log canonical place.
By~\cites{PS01, Bir16a}, we know that for every $n$-dimensional klt singularity
$x\in X$, we may find a boundary $B\geq 0$ on $X$ so that
$(X,B)$ is log canonical, it is not klt, and $m(K_X+B)\sim 0$,
where $m$ is a constant only depending on $n$.
Indeed, while the statement of \cite{Bir16a}*{Theorem 1.8} does not mention it explicitly, the complements constructed there are strictly log canonical, see Step 1 in the proof of \cite{Bir16a}*{Proposition 8.1}.
Hence, a $n$-dimensional exceptional klt singularity admits
a $\frac{1}{m}$-plt blow-up in the sense of~\cite{Mor18b}.
From the proof of Theorem~\ref{thm:plt-loop-commutes}, it follows that there is an exact sequence
\[
1\rightarrow \mathbb{Z}/k\mathbb{Z} \rightarrow \pi_1^{\rm loc}(X,x) \rightarrow N(n) \rightarrow 1, 
\]
where $N(n)$ is a group whose order is bounded by a constant that only depends on $n$.
Moreover, if we assume that $(X,B)$ is $\epsilon$-log canonical for some $\epsilon>0$,
and argue as in the proof of~\cite{Mor18b}*{Theorem 1}, it follows that
$k$ is bounded by a constant only depending on $n$ and $\epsilon$.
The boundedness statement in the exceptional case is also proved in~\cite{HLS19}.
}
\end{example}

\begin{example}\label{ex:toric-singularities}{\em 
Let $N$ be a finitely generated free abelian group of rank $n$,
and let $M={\rm Hom}_\zz(N,\zz)$ its dual.
We denote by $N_\qq$ and $M_\qq$ the corresponding $\qq$-vector spaces.
Let $\sigma\subset N_\qq$ be a rational polyhedral cone, 
and $\sigma^\vee \subset M_\qq$ its dual, i.e.,
all the elements of $M_\qq$ whose pairing with $\sigma$ is non-negative.
Assume that $\sigma$ is full-dimensional.
The affine toric variety 
\[
X(\sigma) \coloneqq {\rm Spec}\left(
\cc[\sigma^\vee \cap M]
\right)
\]
comes with a natural effective $(\cc^*)^n$-action,
and the maximal ideal corresponds to a closed unique fixed point.
The whole geometry of $X(\sigma)$ can be recovered from the combinatorics of $\sigma$.
For instance, $X(\sigma)$ is $\qq$-factorial if and only if 
the extremal rays of $\sigma$ are $\qq$-linearly independent.
We denote by $N_\sigma \subset N$ the lattice generated 
by $\delta \sigma = \sigma \setminus {\rm int}(\sigma)$ in $N$, where ${\rm int}(\sigma)$ denotes the relative interior of the cone, i.e., 
the interior of $\sigma$ in the $\qq$-vector subspace of $N_\qq$ spanned by elements in $\sigma$.
In this case, we have an isomorphism
\[
\pi_1^{\rm loc}(X(\sigma),x) \simeq 
N/N_\sigma.
\]
See, for instance~\cite{CLS11}*{12.1.10}.
On the other hand, if we denote by $\sigma_1 \subseteq \sigma$ the subfan consisting of the extremal rays, and by $N_1 \subset N$ the sublattice generated by $N_1$, we get an isomorphism \[
\pi_1^{\rm reg}(X(\sigma),x) \simeq 
N/N_1.
\]
Hence, the local as well as the regional fundamental group is simply a finite abelian group of rank at most $n$.
}
\end{example}

\begin{example}\label{ex:complexity-one}{\em 
Let $X$ be a $n$-dimensional affine klt variety
with a torus action of dimension $n-1$.
In~\cite{LLM19}, the authors give a description of the local fundamental group of such singularities in terms of isotropy groups of the action.
In the case of quasi-homogeneous log terminal surface singularities, we get a presentation
\[
\langle 
b_1,b_2,t \mid [b_1,t],[b_2,t],t^{e_1}b_1^{m_1},
t^{e_2}b_2^{m_2},t^{e_3}b^{m_3}
\rangle.
\]
Here, $(m_1,m_2,m_3)$ is a platonic triple, 
$\sum_{i=1}^3 \frac{e_i}{m_i}>0$,
$[b_1,t]$ is the commutator of $b_1$ and $t$,
and $b=(b_1b_2)^{-1}$, see~\cite{LLM19}*{Example 4.1}.
In this case, we can take the abelian subgroup to be
the one generated by $t$.
By doing so, the quotient of the local fundamental group 
by the abelian subgroup is just the local fundamental group
of a Du Val singularity, see~\cite{LLM19}*{Example 4.2}.
}
\end{example}

\begin{example}
\label{ex:profinite-completion-not-finite}
{\em 
It is well known that the local fundamental group of a non-klt singularity may be infinite.
In~\cite{KK14}, the authors proved that for every finitely presented group $G$, there is an isolated $3$-fold complex singularity $x\in X$ so that 
$\pi_1^{\rm loc}(X,x)\simeq G$.
\\
It is not hard to show that already for log canonical singularities there are examples with infinite $\pi_1^{\rm loc}(X,x)$.
Indeed, let $c\in C$ be the vertex of the cone over an elliptic curve.
Then, $\pi^{\rm loc}_1(C,c)\simeq \zz^2$.
}
\end{example}

We finish this section by stating some open questions about the local fundamental group, that are motivated by this work.

\begin{question}
{\em In Theorem~\ref{introthm:center}, we prove that the loop $\gamma_\pi$ induced by the plt blow-up is contained in the center subgroup.
Hence, we may define the normal subgroup
$P(x\in X)\subset \pi_1^{\rm loc}(X,x)$ generated by the loops
around exceptional divisors of plt blow-ups.
It is natural to ask whether $P(x\in X)$ has bounded index (depending on the dimension) inside $Z(\pi_1^{\rm loc}(X,x))$, i.e., 
is the center always almost generated by the loops induced by plt blow-ups?}
\end{question}

\begin{question}{\em 
Since two-dimensional klt singularities are quotient singularities, the constant $c'(2)$ equals $60$ 
(see, e.g.~\cite{Col07}*{Proposition C}).
It would be interesting to get an estimate for $c'(3)$.
More generally, it would be interesting to study what invariants of a klt singularity control $c'(n)$.
In general, to control the constant $c'(n)$, one needs to control the corresponding constant for terminal Fano varieties of one dimension less.
Hence, in order to find $c'(3)$ one needs to find the Jordan constant for del Pezzo surfaces.
}
\end{question}

\end{document}